\documentclass[11pt]{article}

\usepackage{amssymb,amsmath,amsthm,setspace}
\usepackage[top=2cm, bottom=2cm, left=2cm, right=2cm]{geometry}
\usepackage[usenames,dvipsnames]{color}
\newcommand{\url}{}
\usepackage[title,titletoc]{appendix}

\theoremstyle{plain}

\newtheorem{thm}{THEOREM}[section]
\newtheorem{lm}[thm]{LEMMA}
\newtheorem{cl}[thm]{COROLLARY}
\newtheorem{prop}[thm]{PROPOSITION}
\theoremstyle{definition}

\theoremstyle{definition}

\newcommand{\R}{{\mathord{\mathbb R}}}

\renewcommand{\|}{{\Vert}}

\numberwithin{equation}{section}
\begin{document}

\title{A Stability Result for a Family of Sharp Gagliardo-Nirenberg Inequalities}
\author{Francis Seuffert\thanks{Work partially supported by U.S. National Science Foundation grants DMS 1201354 and DMS 1501007} \\ Rutgers University}

\maketitle

\section{Introduction}

\begin{abstract}
In their paper, \cite{CaFi}, E. Carlen and A. Figalli prove a stability estimate - also known as a quantitative inequality - for a sharp Gagliardo-Nirenberg inequality and use this result to solve a Keller-Segal Equation. The Gagliardo-Nirenberg Inequality that Carlen and Figalli prove their stability estimate for is part of a larger family of sharp Gagliardo-Nirenberg inequalities, for which the sharp constants and extremals were identified in some cases by Carrillo and Toscani in \cite{CaTo} and then for the remaining cases by Del Pino and Dolbeault in \cite{DeDo}. We prove a stability estimate for the entire family of sharp Gagliardo-Nirenberg inequalities for which Del Pino and Dolbeault calculated the sharp constants and extremals. Establishing stability estimates of inequalities is an active topic. A key piece of our proof is the application of a continuous dimension generalization of the Bianchi-Egnell Stability Estimate proven by the author in \cite{Se}. To see other examples of stability estimates, see \cite{CiFi, CiFu, DoTo, FiMa, FuMa}. In particular, applications of stability estimates to analysis of PDEs can be found in \cite{CiFi, CaFi}.
\end{abstract}

In this paper, we will prove a stability estimate - also known as a quantitative inequality - on a family of sharp Gagliardo-Nirenberg inequalities for which the extremals and sharp constants were proved in part by Carrillo and Toscani in \cite{CaTo} and then the extremals and sharp constants were completely proved later by Del Pino and Dolbeault in \cite{DeDo}. This work is an extension of a stability estimate proved by E. Carlen and A. Figalli in \cite{CaFi}, for a specific Gagliardo-Nirenberg inequality in the family of sharp Gagliardo-Nirenberg inequalities for which we derive a stability estimate. Carlen and Figalli used their stability estimate to help solve a Keller-Segal equation. They proved their stability estimate by exploiting a connection between the Sobolev Inequality and the family of sharp Gagliardo-Nirenberg inequalities that we study. In particular, they use the Bianchi-Egnell Stability estimate, a stability estimate of the Sobolev Inequality, to derive their stability estimate of a Gagliardo-Nirenberg inequality. In proving our extension of Carlen and Figalli's stability estimate, we follow a similar process. However, we use apply continuous dimension extensions of the Sobolev Inequality and the Bianchi-Egnell Stability Estimate in order to derive the stability estimate on the full class of Gagliardo-Nirenberg inequalities of interest.

The pivotal result on stability of the Sobolev Inequality is that of Bianchi and Egnell in 1990, \cite{BiEg}. Since then, there have been many stability estimates: see \cite{CiFu, DoTo, FiMa}, for stability estimates involving Sobolev inequalities and \cite{CaFi, DoTo, FuMa} for stability estimates on other inequalities. Applications of stability estimates have included analysis of PDEs; see \cite{CaFi, CiFi}. In particular, the work in \cite{CiFi} represents new perspectives on applying stability estimates on the Sobolev Inequality to the analysis of PDEs.

A key piece of machinery in the development of this stability estimate is a continuous dimension extension of Bianchi and Egnell's stability estimate of the Sobolev Inequality; this extension was proved in \cite{Se}. We derive our stability estimate of a family of sharp Gagliardo-Nirenberg inequalities from the extension of the Bianchi-Egnell Stability Estimate in \cite{Se}. The central idea connecting these two stability estimates is the connection between the Sobolev Inequality to the Gagliardo-Nirenberg inequalities. To this end, we use this introduction to give some background of these two inequalities, present their connection, and then explain the need for the extension of the Sobolev Inequality and the Bianchi-Egnell Stability Estimate to continuous dimensions in order to prove the stability estimate on the full class of sharp Gagliardo-Nirenberg inequalities of interest.

\subsection{The Sobolev Inequality}

The Sobolev Inequality, which  has been  pivotal in many studies of PDEs, has the following sharp form:
\begin{thm}\label{SobInThm}
Let $N \geq 3$ and let $\varphi\in L^1_{\rm loc}(\R^N)$ have a distributional gradient $\nabla \varphi$ that is square integrable, and suppose also that $\varphi$ vanishes at infinity in the sense that for all $\epsilon>0$, the set $\{ x\in \R^N \: \ |\varphi(x)| > \epsilon\}$ has finite Lebesgue measure. Then 
\begin{equation}\label{SobIn}
\| \varphi \|_{2^*} \leq S_N \| \nabla \varphi \|_2  \,,
\end{equation}
where $2^* = \frac{2N}{N-2}$ and 
\[
S_N = \| F_{1,0} \|_{2^*} / \| \nabla F_{1,0} \|_2 \,,
\]
for
\begin{equation}\label{SobEx}
F_{s, x_0} (x) := s^\frac{N-2}{2} (1 + s|x-x_0|^2)^{\frac{N-2}{2}} \,,
\end{equation}
where $s > 0$, $x_0 \in \mathbb{R}^N$. There is equality if and only if $\varphi$ is a constant multiple of $F_{s, x_0} $ for some
$s > 0$, $x_0 \in \mathbb{R}^N$.
\end{thm}

\noindent This sharp form of the inequality, in which not only the sharp constant, but all of the cases of equality are given, is the result of many individual contributions. 

The proof of the inequality (\ref{SobIn}) with an inexplicit constant  is usually credited to the 1938 paper \cite{So} of S.L. Sobolev.  In hindsight, the Sobolev Inequality with the sharp constant (but without the determination of the cases of equality) can be found in work \cite{Bl}  by G.A. Bliss from 1930. He only considered radial functions, but combining his results with the Faber-Krahn inequality, one would have the inequality in the form stated above. To be more precise,  Bliss's result includes a specification of  the extremals within the class of radial monotone decreasing functions. Combining Bliss' result with the Faber-Krahn rearrangement inequality (known at the time) would have yielded the sharp Sobolev Inequality. More recent work of Brothers and Ziemer on the cases of equality in the Faber-Krahn inequality would also provide the specification of extremals. However, this was not the historical route to the full result in Theorem \ref{SobInThm}.

PDE applications of the Sobolev Inequality did not appear until after Sobolev's 1938 paper, \cite{So} when Sobolev's paper began to be cited in early 1950s. These early applications \cite{Ag,La,Fr} of the Sobolev Inequality to PDEs did not make use of the sharp constant or extremals of the inequality. To our knowledge, the first application of the Sobolev Inequality with an explicit constant is by H. Fujita  in 1961 and then shortly after by Fujita and T. Kato. Fujita \cite{Fuj} used the Sobolev Inequality with an explicit constant to prove existence of weak steady-state solutions of the Navier-Stokes equation provided the boundary data satisfy an explicit smallness condition. In their 1964 paper, \cite{FuKa}, Fujita and Kato used the Sobolev Inequality with an explicit constant to prove global existence in time of solutions of the Navier-Stokes equation with suitably small initial conditions. 

The Sobolev constant is what determines the explicit form of the smallness condition in these papers. This fact motivated the search for an explicit value. Fujita \cite{Fuj} showed that the sharp Sobolev constant for three dimensions, $S_3$, satisfies
\begin{equation}\label{ShConsBd}
S_3 \leq (4/\pi)^{1/3} \,.
\end{equation}
The next major development in calculating a sharp constant is from a 1971 paper by R. Rosen, \cite{Ro}, in which it is shown that the sharp constant for $N=3$ is no bigger than
\begin{equation}\label{ShCons3Dim}
4^{1/3}/ (3 \pi)^{1/2} \,;
\end{equation}
this value is in fact sharp. Finally, in a paper published in 1976, \cite{Au}, T. Aubin computed the sharp constant for the Sobolev Inequality and classified its extremals. Slightly later in 1976, G. Talenti also published a paper calculating the sharp constant and classifying the extremals. Aubin's work is set in a geometric investigation of the isoperimetric inequality, whereas Talenti's proof is deduced in a purely analytic argument.

\subsection{Deriving a Family of Sharp Gagliardo-Nirenberg Inequalities from the Sobolev Inequality: a Partial Result}

A functional inequality that is profoundly related to the Sobolev Inequality is a family of sharp Gagliardio-Nirenberg inequalities, GN inequalities for short, whose sharp constant and extremals were fully classified by Del Pino and Dolbeault, and earlier classified in part by Carillo and Toscani; see Carrillo and Toscani's work, \cite{CaTo}, for applications of the GN inequalities to the porous medium equation. Indeed, one can derive the sharp form of the entire family of these GN inequalities and deduce their extremals from an extension of the Sobolev Inequality to continuous dimensions. We will show how to derive some of these sharp GN inequalities with the corresponding extremals from the Sobolev Inequality in a little bit. Before doing so, we introduce these sharp GN inequalities: Let $u \in \dot{H}^1 (\mathbb{R}^n)$ for $n \geq 2$. Then, for $1 \leq t \leq n/(n-2)$ (if $n=2$, $1 \leq t < \infty$)
\begin{equation}\label{GN ineq's}
\| u \|_{2t} \leq A_{n,t} \| \nabla u \|_2^\mu \| u \|_{t+1}^{1 - \mu} \,, \mu = \frac{n(t-1)}{t[2n - (t+1)(n-2)]} \,,
\end{equation}
where $A_{n,t}$ is a sharp constant depending on $n$ and $t$. One should note that when $t = 1$, $\mu = 0$, and (\ref{GN ineq's}) is a trivial inequality; and when $t = \frac{n}{n-2}$, $\mu = 1$, and (\ref{GN ineq's}) is the Sobolev Inequality. Also, $\mu$ varies continuously between 0 and 1 as $t$ varies between 1 and ${2^*}$ (or $\infty$ if $n = 2$). Moreover, the complete set of extremals of the one-parameter set of GN inequalities given in (\ref{GN ineq's}) is the constant multiples of the functions given by
\begin{equation}\label{Extremals}
v_{\lambda, x_0} (x) := \lambda^\frac{n}{2t} (1 + \lambda^2 |x - x_0|^2)^{-1/(t-1)} \,, \lambda > 0 \,, x_0 \in \mathbb{R}^n \,.
\end{equation}
For convenience, we define the following function:
\begin{equation}\label{v Def}
v(x) := v_{1,0} (x) \,.
\end{equation}
Note that the above implies that
\begin{equation}\label{SharpConst}
A_{n,t} = \| v \|_{2t} / \| \nabla v \|_2^\mu \| v \|_{t+1}^{1-\mu} \,.
\end{equation}

The main result of this paper is the derivation of a stability estimate on the sharp GN inequalities summarized by (\ref{GN ineq's}) from a stability estimate on an extension of the Sobolev Inequality to continuous dimensions. The key to deriving this relationship, is a striking observation that one can use an extension of the Sobolev Inequality to continuous dimensions to derive the sharp GN inequalities, (\ref{GN ineq's}), and their extremals. This relationship is stated and illustrated in the work of Bakry, Gentil, and Ledoux in \cite{BaGe}. For the time being, we will present a partial result, using only the classical Sobolev Inequality in integer dimensions, to illustrate this connection. To this end, we present and prove the following

\begin{prop}\label{Sob to GN}
The Sobolev Inequality in $(m+n)$-dimensions, for integers $m >0$ and $n \geq 2$, can be used to deduce the sharp GN inequalities, (\ref{GN ineq's}), and their extremals for $t$ given by
\begin{equation}\label{tFormula}
t = \frac{m+2n}{m+2n-4} \,.
\end{equation}
\end{prop}

\begin{proof}
Let $\varphi: \mathbb{R}^{m+n} \to \mathbb{R}$ be given by
\begin{equation}\label{Phif}
\varphi (x,y) = [f(x) + |y|^2]^{-(m+n-2)/2} \,, x \in \mathbb{R}^n, y \in \mathbb{R}^m,
\end{equation}
for $f \geq 0$. Then
\begin{align}
|\nabla_{x,y} \varphi|^2 &= \left( \frac{m+n-2}{2} \right)^2 [ f + |y|^2 ]^{-(m+n)} (|\nabla_x f|^2 + 4|y|^2) \,, \text{ and} \nonumber \\
|\varphi|^{2^*} &= [f + |y|^2]^{-(m+n)} \,. \nonumber
\end{align}
Thus, the integrated Sobolev Inequality takes the form
\begin{equation}\label{IntSob}
\left( \int_{\mathbb{R}^n} f^{-\frac{m+2n}{2}} \mathrm{d}x \right)^{(m+n-2)/(m+n)} \leq c_1 \int_{\mathbb{R}^n} f^{-\frac{m+2n}{2}} |\nabla_x f|^2 \mathrm{d}x + c_2 \int_{\mathbb{R}^n} f^{-\frac{m+2n-2}{2}} \mathrm{d}x \,,
\end{equation}
for constants $c_1$ and $c_2$ depending upon $m$ and $n$. Since we obtained (\ref{IntSob}) by integrating the Sobolev Inequality applied to (\ref{Phif}) in the $y$-variable, (\ref{IntSob}) yields equality if $f(x) = 1 + |x|^2$, because this $f$ makes $\varphi$ given by (\ref{Phif}) into an extremal for the Sobolev Inequality. Replacing $f$ with $u^{-\frac{4}{m+2n-4}}$, (\ref{IntSob}) becomes
\begin{equation}\label{GN Add}
\| u \|_{2t}^{4t/2^*} \leq c_3 \| \nabla u \|_2^2 + c_4 \| u \|_{t+1}^{t+1} \,,
\end{equation}
for further constants $c_3$ and $c_4$, and $t$ given by (\ref{tFormula}). If we replace $u$ with $u_\lambda$ given by
\[
u_\lambda (y) = \lambda^{n/2t} u (\lambda y) \,,
\]
in (\ref{GN Add}) and then optimize with respect to $\lambda$, we find that for $u$ such that
\[
\| \nabla u \|_2^2 / \| u \|_{t+1}^{t+1} = \| \nabla v \|_2^2 / \| v \|_{t+1}^{t+1} \,,
\]
(\ref{GN Add}) becomes (\ref{GN ineq's}), because
\[
c_3 \| \nabla u \|_2^2 + c_4 \| u \|_{t+1}^{t+1} = A_{n,t}^{4t/2^*} \| \nabla u \|_2^{(4t/2^*) \mu} \| \nabla u \|_{t+1}^{(4t/2^*) (1-\mu)} \,.
\]

Finally, we deduce that there is equality in (\ref{GN ineq's}) if $u = v$, because in this case, $f (x) = v^{-\frac{4}{m+2n-4}} (x) = 1+|x|^2$. In this case, $\varphi$ given by (\ref{Phif}) is the Sobolev extremal, $F_{1,0}$, and we must have equality. Having concluded that $v$ is an extremal of (\ref{GN ineq's}), we can deduce the remaining extremals as a result of translation invariance of the norms and homogeneity on each side of the inequality (\ref{GN ineq's}) with respect to constant multiples and conformal invariance with respect to dilation characterized by the operation:
\[
u (x) \mapsto \lambda^\frac{n}{2t} u(\lambda x) \,.
\]
\end{proof}

\subsection{Motivation for Our Main Results: Deriving Gagliardo-Nirenberg Stability Estimates from Sobolev Stability Estimates}

In 1983 H. Brezis and L. Nirenberg wrote a paper, \cite{BrNi}, chronicling their investigation of positive solutions of nonlinear elliptic equations involving critical Sobolev exponents. In the process of this investigation, they proved some improved Sobolev inequalities. A specific case of the general class of elliptic equations that Brezis and Nirenberg analyze is given by
\begin{align}\label{BrNiEE2}
-\Delta \varphi =& \varphi^{2^*-1} + \lambda \varphi \,, \text{ on } \Omega \nonumber \\
\varphi >& 0 \,, \text{ on } \Omega \nonumber \\
\varphi =& 0 \,, \text{ on } \partial \Omega \,,
\end{align}
for $\Omega \subseteq \mathbb{R}^N$ a bounded domain with $N \geq 3$. For $N = 3$, Brezis and Nirenberg proved the following
\begin{thm}\label{BrNiThm1}
Let $\lambda_1$ be the lowest eigenvalue of $-\Delta$ with zero Dirichlet condition on a bounded domain $\Omega \subseteq \mathbb{R}^N$ with $N \geq 4$. Then for every $\lambda \in (0,\lambda_1)$, there exists a solution of (\ref{BrNiEE2}).
\end{thm}
Brezis and Nirenberg used Theorem \ref{BrNiThm1} to prove the following improvement on the Sobolev Inequality:
\begin{cl}\label{SobInImpCl}
Assume $\Omega \subseteq \mathbb{R}^3$ is a bounded domain. Then, there exists $\lambda^*$, $0<\lambda^*<\lambda_1$, such that
\begin{equation}\label{SobInImp1}
\| \nabla \varphi \|_2^2 \geq S_3^{-2} \| \varphi \|_6^2 + \lambda^* \| \varphi \|_2^2 \,, \text{ } \forall \varphi \in H_0^1 (\Omega) \,.
\end{equation}
We may take $\lambda^* = \frac{1}{4} (3 |\Omega| / 4 \pi)^{-2/3}$, where $| \cdot |$ denotes Lebesgue measure. This value is sharp when $\Omega$ is a ball.
\end{cl}
\noindent Brezis and Nirenberg proved a theorem quite similar to Theorem \ref{BrNiThm1} for the case $N \geq 4$, but they did not use this theorem to prove an improved Sobolev Inequality, like in the case of $N = 3$. Instead, they employed a lemma of E. Lieb's and Pohozaev's Identity to deduce the following improved Sobolev inequality:
\begin{equation}\label{Imp1}
\| \nabla \varphi \|_2^2 \geq S_N^{-2} \| \varphi \|_{2^*}^2 + \lambda_p (\Omega) \| \varphi \|_p^2 \,, \text{ } \forall \varphi \in H_0^1 (\Omega) \,,
\end{equation}
for $\Omega \subseteq \mathbb{R}^N$ a bounded domain, $1 \leq p < 2^*/2$, with $\lambda_p (\Omega)$ a constant depending only on $p$, $N$, and $\Omega$ and $\lambda_p (\Omega) \to 0$ as $p \to 2^*/2$. In the following year, 1984, Brezis and E. Lieb proved some further improvements upon the Sobolev Inequality for bounded domains. They also posed an important question in the direction of improving the Sobolev Inequality: ``Is there a natural way to bound $S_N^2 \| \nabla \varphi \|_2^2 - \| \varphi \|_{2^*}^2$ from below in terms of the `distance' of $\varphi$ from the set of [extremals given by (\ref{SobEx})].''

In 1990, G. Bianchi and H. Egnell came up with a strong positive answer to Brezis and Lieb's question, in the form of the following stability estimate:
\begin{thm}[Bianchi-Egnell Stability Estimate]\label{BEthm}
There is a positive constant, $\alpha$, depending only on the dimension, $N \geq 3$, so that
\begin{equation}\label{BEStabEst}
S_N^2 \| \nabla \varphi \|_2^2 -  \| \varphi \|_{2^*}^2 \geq \alpha d(\varphi,M)^2 \,,
\end{equation}
$\forall \varphi \in \dot{H}^1 (\mathbb{R}^N)$, where $d (\cdot, M)$ is the distance functional given by
\begin{equation}\label{DistFunc}
d(\varphi, M) = \inf_{z \in \mathbb{R}, s > 0, x_0 \in \mathbb{R}^N} \| \nabla (\varphi - z F_{s,x_0}) \|_2 \,.
\end{equation} Furthermore, the result is sharp in
the sense that it is no longer true if $d(\varphi,M)^2$ in (\ref{BEStabEst}) is replaced with $d(\varphi,M)^\beta \| \nabla \varphi \|_2^{2 - \beta}$, where $\beta < 2$.
\end{thm}
\noindent As mentioned above, (\ref{BEStabEst}) is a stability estimate, which is an inequality that bounds the difference in terms of a sharp inequality from below by the distance of a given function from the extremals of the sharp inequality. In a previous paper, \cite{Se}, the author proved a continuous dimension extension of Bianchi and Egnell's stability estimate. This extension was proved in order to deduce a further stability estimate on the family of sharp Gagliardo-Nirenberg inequalities summarized in (\ref{GN ineq's}).

An open problem associated with the Bianchi-Egnell Stability Estimate is the calculation of an explicit constant $\alpha$ that satisfies (\ref{BEStabEst}). All we know is that there is some $\alpha > 0$ for which (\ref{BEStabEst}) is true. This is because the process of obtaining the stability estimate involves finding a local stability estimate and then applying concentration compactness to prove that (\ref{BEStabEst}) must hold for some $\alpha > 0$. The extension of the Bianchi-Egnell Stability Estimate to continuous dimensions suffers from the same problem. However, following a similar process to the proof of the original Bianchi-Egnell Stability Estimate, we prove a local Bianchi-Egnell Stability Estimate, that is in fact more quantitative in nature than Bianchi and Egnell's original local stability estimate. This is because in our proof of our local stability estimate, we use an argument that gets explicit bounds on the remainder of the second order Taylor expansion of the difference of terms in the Sobolev Inequality. Obtaining an explicit constant in our extension of the Bianchi-Egnell Stability Estimate to continuous dimensions would be useful for applications, because it would yield an explicit constant for the full class of sharp Gagliardo-Nirenberg inequalities, (\ref{GN ineq's}). E. Carlen and A. Figalli used a special case of this stability estimate to help solve a Keller-Segel Equation. If we had an explicit constant for the stability estimate for the full class of Gagliardo-Nirenberg inequalities of Del Pino and Dolbeault, it could be useful for more PDE applications.

In this paper, we extend Carlen and Figalli's stability estimate to the full family of GN inequalities, (\ref{GN ineq's}). The GN inequality that Carlen and Figalli derive a stability estimate for is the following:
\begin{equation}\label{GN ineq CaFi}
\| u \|_6 \leq \pi^{-1/6} \| \nabla u \|_2^{1/3} \| u \|_4^{2/3} \,,
\end{equation}
for all $u \in \dot{H}^1 (\mathbb{R}^2)$. If we subtract the left hand side from the right hand side of (\ref{GN ineq CaFi}), the resulting quantity is larger than or equal to zero. Carlen and Figalli improve upon this, by showing that for a normalized nonnegative $u$, this difference controls the distance of $u$ from a subset of the extremals of (\ref{GN ineq CaFi}). To make this notion precise, we define the GN deficit functional, $\delta_{GN} [ \cdot ]$:
\begin{equation}\label{GN Deficit CaFi}
\delta_{GN} [u] := \| \nabla u \|_2 \| u \|_4^2 - \pi^{1/2} \| u \|_6^3 \,.
\end{equation}
The stability estimate of Carlen and Figalli is summarized in the following
\begin{thm}\label{CaFiStabEst}
Let $u \in \dot{H}^1 (\mathbb{R}^2)$ be a nonnegative function such that $\| u \|_6 = \| v \|_6$. Then there exist universal constants $K_1, \delta_1 > 0$ such that whenever $\delta_{GN} [u] \leq \delta_1$,
\begin{equation}\label{CaFiIneq}
\inf_{\lambda > 0, x_0 \in \mathbb{R}^2} \| u^6 - \lambda^2 v_{\lambda, x_0} \|_1 \leq K_1 \delta_{GN} [u]^{1/2} \,.
\end{equation}
\end{thm}

The stability estimate that Carlen and Figalli proved is for the inequality (\ref{GN ineq's}) with $n = 2$ and $t = 3$. This leads to the natural question, ``does Carlen and Figalli's stability estimate extend to the entire family of sharp GN inequalities, (\ref{GN ineq's})?'' The answer to this question is yes, but proving this extension in a manner similar to Carlen and Figalli's proof required an extension of Bianchi and Egnell's stability estimate of the Sobolev Inequality to continuous dimension. We will outline Carlen and Figalli's proof of Theorem \ref{CaFiStabEst} in order to clarify the need for the extension of Bianchi and Egnell's Stability Estimate.

The key to proving Carlen and Figalli's GN stability estimate is exploring the connection between the sharp GN inequalities, (\ref{GN ineq's}), and the Sobolev Inequality summarized in Proposition \ref{Sob to GN} for the dimensions $m = n = 2$. In particular, Carlen and Figalli derive explicit formulas for $u \in \dot{H}^1 (\mathbb{R}^2)$ and $\varphi \in \dot{H}^1 (\mathbb{R}^4)$ such that the difference of terms in the Sobolev Inequality applied to $\varphi$ equals the difference in terms of the GN inequality applied to $u$. Their exact statement is summarized in the following
\begin{prop}\label{CaFiDeficits}
Let $u \in \dot{H}^1 (\mathbb{R}^2)$ be a nonnegative function such that
\[
\| u \|_4^4 / \| \nabla u \|_2^2 = \| v \|_4^4 / \| \nabla v \|_2^2 = 1 / 2 \,.
\]
Let $\varphi_u: \mathbb{R}^4 \to \mathbb{R}$ be given by
\[
\varphi_u (y,x) = [ f (x) + |y|^2 ]^{-1} \,, \text{ } f(x) = u^{-2} (x) \,.
\]
Then
\begin{equation}\label{CaFiDeficitEq}
\sqrt{3} \left( \frac{1}{4 \pi} \sqrt{\frac{3}{2}} \| \nabla \varphi_u \|_2^2  - \| \varphi_u \|_4^2 \right) = \| \nabla u \|_2 \| u \|_4^2 - \pi^{1/2} \| u \|_6^3 =: \delta_{GN} [u] \,.
\end{equation}
\end{prop}

\begin{proof}
We compute
\begin{align}
\| \nabla \varphi_u \|_2^2 &= \int_{\mathbb{R}^2} \left( \int_{\mathbb{R}^2} \frac{|\nabla f(x)|^2}{(f(x) + |y|^2)^4} \mathrm{d}y \right) \mathrm{d}x + \int_{\mathbb{R}^2} \left( \int_{\mathbb{R}^2} \frac{4|y|^2}{(f(x) + |y|^2)^4} \mathrm{d}y \right) \mathrm{d}x \nonumber \\
&= \frac{\pi}{3} \int_{\mathbb{R}^2} |\nabla f(x)|^2 f^{-3} (x) \mathrm{d}x + \frac{2 \pi}{3} \int_{\mathbb{R}^2} f^{-2} (x) \mathrm{d}x \nonumber
\end{align}
and
\[
\| \varphi_u \|_4^2 = \left( \frac{\pi}{3} \int_{\mathbb{R}^2} f^{-3} (x) \mathrm{d}x \right)^{1/2} \,.
\]
Thus,
\[
0 \leq \frac{1}{4 \pi} \sqrt{\frac{3}{2}} \| \nabla \varphi_u \|_2^2 - \| \varphi_u \|_4^2 = \frac{1}{2 \sqrt{6}} \left( 2 \int_{\mathbb{R}^2} | \nabla u |^2 \mathrm{d}x + \int_{\mathbb{R}^2} u^4 \mathrm{d}x \right) - \left( \frac{\pi}{3} \int_{\mathbb{R}^2} u^6 \mathrm{d}x \right) \,,
\]
or equivalently, using the identity $2 \sqrt{AB} = A + B - (\sqrt{A} - \sqrt{B})^2$ and writing all integrals as norms,
\begin{align}
& \| \nabla u \|_2 \| u \|_4^2 - \pi^{1/2} \| u \|_6^3 = \sqrt{3} \left( \frac{1}{4 \pi} \sqrt{\frac{3}{2}} \| \nabla \varphi_u \|_2^2 - \| \varphi_u \|_4^2 \right) - \frac{1}{2 \sqrt{2}} \left( \sqrt{2} \| \nabla u \|_2 - \| u \|_4^2 \right) \,. \nonumber
\end{align}
Recalling the assumption that $\| \nabla u \|_2^2/ \| u \|_4^4 = 1/2$, we conclude (\ref{CaFiDeficitEq}).
\end{proof}

At this point, Carlen and Figalli combine Proposition \ref{CaFiDeficits}, the Bianchi-Egnell Stability Estimate, and the Sobolev Inequality as follows:
\begin{align}\label{CaFiKey}
\delta_{GN} [u] =& \| \nabla u \|_2 \| u \|_4^2 - \pi^{1/2} \| u \|_6^3 \nonumber \\
=& \sqrt{3} \left( \frac{1}{4 \pi} \sqrt{\frac{3}{2}} \| \nabla \varphi_u \|_2^2  - \| \varphi_u \|_4^2 \right) \,, \text{ by Proposition \ref{CaFiDeficits}} \nonumber \\
\geq& \alpha \sqrt{3} \inf_{z \in \mathbb{R}, s > 0, (x_0, y_0) \in \mathbb{R}^2 \times \mathbb{R}^2} \| \nabla (\varphi - z F_{s, (x_0, y_0)} ) \|_2^2 \,, \text{ by the Bianchi-Egnell Stability Estimate} \nonumber \\
\geq& \alpha \frac{32 \pi^2}{\sqrt{3}} \inf_{z \in \mathbb{R}, s > 0, (x_0, y_0) \in \mathbb{R}^2 \times \mathbb{R}^2} \| \varphi - z F_{s, (x_0, y_0)} \|_4^2 \,, \text{ by the Sobolev Inequality.}
\end{align}
Carlen and Figalli then bridge the gap between (\ref{CaFiKey}) and their stability estimate in a series of two propositions and a change of variables. We briefly summarize these steps in the following paragraphs.

Note that the normalization, $\| u \|_6 = \| v \|_6$, which is a condition of Carlen and Figalli's stability estimate, is equivalent to $\| \varphi_u \|_4 = \| F_{1,(0,0)} \|_4$. This allows Carlen and Figalli to employ the following
\begin{lm}\label{UsefulLm1}
Let $\varphi$ be given by $\varphi(x,y) = [f(x) + |y|^2]^{-1}$, with $f: \mathbb{R}^2 \to \mathbb{R}$ nonnegative and $F_{1,(0,0)}$ given by (\ref{SobEx}). Suppose that $\| \varphi \|_4 = \| F_{1,(0,0)} \|_4$. Then, there is a universal constant $C_1$ so that for all real numbers $\delta>0$ with
\[
\delta^{1/2} \leq 2400^{-1} \,,
\]
whenever
\[
\| \varphi - z F_{s,(x_0,y_0)} \|_4 \leq \delta^{1/2} \text{ for some } z, s, x_0, y_0
\]
then
\[
\| \varphi - F_{1,(0,0)} \|_4 \leq C_1 \delta^{1/2} \,.
\]
\end{lm}
\noindent A possible value for $C_1$ is 4800. Next, they prove
\begin{lm}\label{UsefulLm2}
Let $u \in \dot{H}^1 (\mathbb{R}^2)$ be a nonnegative function satisfying
\begin{equation}\label{NormCdn}
\| u \|_6 = \| v \|_6 \text{ and } \| u \|_4^4 / \| \nabla u \|_2^2 = \| v \|_4^4 / \| \nabla v \|_2^2 \,,
\end{equation}
and let $\varphi$ be as defined in Proposition \ref{UsefulLm1}. Supposed that $\| \varphi - F_{1,(0,0)} \| \leq 1$. Then
\[
\| u^6 - v^6 (\cdot - x_0) \|_1 \leq C_2 \| \varphi - F_{1,(0,0)} \|_4
\]
for some constant $C_2$.
\end{lm}
\noindent A possible choice for $C_2$ is 1000.
Thus, if $u \in \dot{H}^1 (\mathbb{R}^2)$ satisfies (\ref{NormCdn}), then we can combine (\ref{CaFiKey}), Lemma (\ref{UsefulLm1}), and Lemma (\ref{UsefulLm2}) to conclude that there exist universal constants $K_1, \delta_1 > 0$ such that, whenever $\delta_{GN} [u] \leq \delta_1$,
\[
\| u^6 - v^6 (\cdot - x_0) \|_1 \leq K_1 \delta_{GN} [u]^{1/2} \,.
\]

Next, $\delta_{GN} [u]$ and $\| u \|_6$ are both unchanged if $u(x)$ is replaced by $u_\sigma := \sigma^{1/3} u (\sigma x)$. Thus, assuming only that $\| u \|_6 = \| v \|_6$, we may choose a scale parameter $\sigma$ such that $\| \nabla u_\sigma \|_2^2 / \| u_\sigma \|_4^4 = 1/2$. Hence, we conclude that
\[
\int_{\mathbb{R}^2} | \sigma^2 u^6 (\sigma x) - v(x - x_0)| \mathrm{d}x \leq K_1 \delta_{GN} [u]^{1/2} \,.
\]
Changing variables once more, and taking $\lambda := 1/\sigma$, we obtain
\[
\int_{\mathbb{R}^2} |u^6 (x) - \lambda^2 v (\lambda x - x_0)| \mathrm{d}x = K_1 \delta_{GN} [u]^{1/2} \,,
\]
which proves (\ref{CaFiIneq}) and concludes the proof of Theorem (\ref{CaFiStabEst}).

\subsection{Main Result}

In this paper, we derive a stability estimate for the full family of inequalities (\ref{GN ineq's}). Roughly speaking, the stability estimate tells us how far away a given function is from the manifold of optimizers for the GN (Gagliardo-Nirenberg) inequalities in terms of its GN deficit, denoted $\delta_{GN} [u]$, given by
\begin{equation}\label{GN Deficit}
\delta_{GN} [u] := A_{n,t}^{4t/2^*} \| \nabla u \|_2^{\mu 4t/2^*}\| u \|_{t + 1}^{(1 - \mu) 4t/2^*} - \| u \|_{2t}^{4t/2^*} \,.
\end{equation}
The complete set of extremals of the one-parameter set of GN inequalities given in (\ref{GN ineq's}) is the constant multiples of the functions given by
\[
v_{\lambda, x_0} (x) = \lambda^\frac{n}{2t} (1 + \lambda^2 |x - x_0|^2)^{-1/(t-1)} \,, \lambda > 0 \,, x_0 \in \mathbb{R}^n \,.
\]

The precise statement of the stability estimate that we prove for (\ref{GN ineq's}) is 
\begin{thm}\label{GN StabEst}
Let $u \in \dot{H}^1 (\mathbb{R}^n)$ be  a nonnegative function such that $\| u \|_{2t} = \| v \|_{2t}$.  Then there exist positive constants $K_1 := K_1 (n, t)$ and $\delta_1 := \delta_1 (n, t)$, depending upon $n$ and $t$, such that whenever $\delta_{GN} [u] \leq \delta_1$,
\begin{equation}\label{StabEst}
\inf_{\lambda > 0, x_0 \in \mathbb{R}^n} \| u^{2t} - v_{\lambda, x_0}^{2t} \|_1 \leq K_1 \delta_{GN} [u]^{1/2} \,.
\end{equation}
\end{thm}
\noindent The proof of this Theorem follows in four parts, which correspond to the sections two through five. The steps used to prove Theorem \ref{GN StabEst} are mostly an adaptation of the steps that Carlen and Figalli used to prove their stability estimate, Theorem \ref{CaFiStabEst}. In the remainder of this introduction, we explain the need for continuous dimension extensions of the Sobolev Inequality and the Bianchi-Egnell Stability Estimate, as well as give complete statements of these extensions. Understanding these continuous dimension extensions is essential for understanding the proof of Theorem \ref{GN StabEst}.

\subsection{Generalizing Carlen and Figalli's Stability Estimate: Difficulties and the Need for an Extension of the Bianchi-Egnell Stability Estimate}

In this paper, we derive a stability estimate for the full family of inequalities (\ref{GN ineq's}). One could try to deduce a stability estimate for the sharp GN inequalities, (\ref{GN ineq's}), directly, perhaps following steps similar to those of Bianchi and Egnell in their proof of their stability estimate of the Sobolev Inequality. Roughly speaking, such a proof would break into two parts. The first part would involve a Taylor expansion of the difference in terms of the GN inequality to the second order at $v$ and getting some sort of estimate on the remainder. To be more precise, we could calculate the the first and second variation at $v$ of the functional
\begin{equation}\label{PreDeficit}
u \in \dot{H}^1 \mapsto A_{n,t}^\beta \| \nabla u \|_2^{\beta \mu} \| u \|_{t+1}^{\beta (1-\mu)} - \| u \|_{2t}^\beta
\end{equation}
for some $\beta > 0$ - there is a precise formula for $\beta$, but it requires some background and the precise form of $\beta$ is not essential for the present argument, so we postpone its formula until later. Calculating these variations is a bit more complicated than the analagous calculations of Bianchi and Egnell, who only had to deal with the quantity $C_N^2 \| \varphi \|_2^2 - \| \varphi \|_{2^*}^2$, which only has two norms of $\varphi$ as opposed to (\ref{PreDeficit}), which has three norms of $u$. Next, after some analysis of (\ref{PreDeficit}) with respect to its first and second variations at $v$ and the remainder terms of (\ref{PreDeficit}) above the second order, we would hopefully attain a local quantitative stability estimate of (\ref{GN ineq's}). Next, to pass from this local quantitative stability estimate to a global one, we could apply a concentration compactness argument. The only trouble, is that concentration compactness arguments for the GN inequalities, (\ref{GN ineq's}), are to our knowledge absent from literature. Thus, we would have to develop a concentration compactness argument for the GN inequalities, (\ref{GN ineq's}). To avoid these difficulties, we modeled our proof of the stability estimate of (\ref{GN ineq's}) off of Carlen and Figalli's approach instead. To be more precise, we exploited a connection between a continuous dimension extension of the Sobolev Inequality and the GN inequalities, (\ref{GN ineq's}), and leveraged an extension to continuous dimensions of Bianchi and Egnell's Stability Estimate to deduce a stability estimate on (\ref{GN ineq's}). We needed extensions to continuous dimensions of the Sobolev Inequality and the Bianchi-Egnell Stability Estimate to provide a bridge to the full family of GN inequalities, (\ref{GN ineq's}). We make this precise in the next paragraph.

The fact that Carlen and Figalli could apply the Bianchi-Egnell Stability Estimate to deduce Theorem \ref{CaFiStabEst} with $n = 2$ and $t = 3$ is a happy coincidence. The key equality, (\ref{CaFiDeficitEq}), that allows Carlen and Figalli to deduce their GN stability estimate from the Bianchi-Egnell Stability Estimate is obtained by constructing $\varphi_u (x,y)$ out of a given $u(x) \in \dot{H}^1  (\mathbb{R}^2)$. Note that the dimension, $n = 2$, for $u(x)$ corresponds to the $x$-variable in $\varphi_u (x,y)$. The fact that the left hand side  of (\ref{CaFiDeficitEq}) corresponds to $\delta_{GN} [u]$ for $t = 3$ is a result of the other variable of $\varphi_u (x,y)$, $y$, being in two dimensions. It is not possible to generalize Proposition \ref{CaFiDeficits} over all $n \geq 2$ and $1 \leq t \leq \frac{n}{n-2}$ (or $< \infty$ if n = 2) for a generalized notion of $\varphi_u (x,y)$ and $\delta_{GN} [u]$ if $y \in \mathbb{R}^m$ for $m$ integer and we model Carlen and Figalli's proof of Proposition \ref{CaFiDeficits}. To be more precise, if we take some nonnegative $u \in \dot{H}^1 (\mathbb{R}^n)$, $n \geq 2$ such that
\[
\| \nabla u \|_2^2 / \| u \|_{t+1}^{t+1} = \| \nabla v \|_2^2 / \| v \|_{t+1}^{t+1}
\]
and set
\[
\varphi_u (x,y) = [ u^{- \frac{4}{m+2n-4}} (x) + |y|^2 ]^{-(m+n-2)/2}
\]
for $y \in \mathbb{R}^2$, then we could deduce that
\begin{equation}\label{PartialDeficits}
\tilde{C} ( C_N^2 \| \nabla \varphi_u \|_2^2 -\| \varphi \|_{2^*}^2 ) = A_{n,t}^{4t/2^*} \| \nabla u \|_2^{(4t/2^*) \mu} \| u \|_{t+1}^{(4t/2^*)(1-\mu)} - \| u \|_{2t}^{4t/2^*}
\end{equation}
where
\begin{equation}\label{Parameters}
t = \frac{m+2n}{m+2n-4} \,, \text{ } 2^* = \frac{2(m+n)}{m+n-2} \,,
\end{equation}
and $\tilde{C}$ is a calculable constant. The issue is that if we restrict ourselves to integer $m$, then we only deduce (\ref{PartialDeficits}) for $t$ according (\ref{Parameters}) with $m$ and $n$ natural numbers and $n \geq 2$ - this is nowhere near the full range of $1 \leq t \leq \frac{n}{n-2}$ (or $t < \infty$ if $n = 2$). Thus, the only way to deduce (\ref{PartialDeficits}) for the full range of $n$ and $t$ is to build $\varphi_u$ with the ``$y$-variable'' part in continuous dimensions. Consequently, to deduce a stability estimate for the full class of sharp GN inequalities, (\ref{GN ineq's}), using Carlen and Figalli's methods, we also need to derive a Bianchi-Egnell stability estimate for $\varphi(x,y)$ with $y$ a continuous dimension variable. We will make these notions precise in the next two subsections.

\subsection{Bakry, Gentil, and Ledoux's Extension of the Sharp Sobolev Inequality with Nguyen's Classification of Extremals}

As mentioned several times now, our approach to generalizing Carlen and Figalli's stability estimate employs an extension to continuous dimensions of Bianchi and and Egnell's Stability Estimate. In order to have a continuous dimension extension of the Bianchi-Egnell Stability Estimate, we need a continuous dimension extension of the Sobolev Inequality with a sharp constant and classification of extremals. We introduce the notion of continuous dimension extension of the Sobolev Inequality appropriate to our needs here. Bakry, Gentil, and Ledoux proved an extension of the Sobolev Inequality in p. 322-323 of \cite{BaGe}. This extension is for ``cylindrically symmetric'' functions on Euclidean space of $m+n$ dimensions, where one of $m$ and $n$ is not necessarily an integer. To state Bakry, Gentil, and Ledoux's extension of the Sharp Sobolev Inequality, we need to define the appropriate norms and spaces. First, we establish some properties of cylindrically symmetric functions. Let $\varphi: \mathbb{R}^{n} \times [ 0, \infty ) \to \mathbb{C}$ be a cylindrically symmetric function.  What we mean when we say that $\varphi$ is a cylindrically symmetric function is that if we write $\varphi$ as $\varphi (x, \rho)$, where $\rho$ is a variable with values in $[0,\infty)$ and $x$ is the standard $n$-tuple on $n$ Cartesian coordinates, that the $\rho$ variable acts as a radial variable in $m$-dimensions while the $x$ variable represents the other $n$-dimensions on which $\varphi$ acts.  If $m$ is an integer, then $\varphi$ would also have a representation as a function on $\mathbb{R}^{m+n}$.  For example,
\[
\varphi (x, \rho) = (1 + |x| + \rho^2)^{-1}
\]
as a cylindrically symmetric function for $m = n = 2$ has the representation as a
function on $\mathbb{R}^4$:
\[
\varphi (x_1, x_2, x_3, x_4) = \left( 1 + \sqrt{x_1^2 + x_2^2} + x_3^2 + x_4^2 \right)^{-1} \,,
\]
where $x_1$ and $x_2$ correspond to the $x$-variable of $\varphi (x, \rho)$ and $x_3$ and $x_4$ correspond to the $\rho$-variable of $\varphi (x, \rho)$.  However, we want $m$ to  also possibly be noninteger.  Note, that the value of $m$ is not provided when we give the equation for $\varphi$.  In this paper, the value of $m$ will be determined by the dimensions over which our norms are integrated.  To be more precise, the $m$ dimensions of Euclidean space are encoded in the measure of integration corresponding to the $\rho$ variable.  This measure is $\omega_m \rho^{m-1} \mathrm{d}\rho$, where $\omega_m$ is a generalized notion of the area of the unit $(m-1)$-sphere given by
\begin{equation}\label{SphArea}
\omega_m := 2 \pi^{m/2} / \Gamma(m / 2) \,,
\end{equation}
note that this formula is valid for $m > 0$. In this case, the $L^p$-norm of $\varphi$ is given by
\[
\| \varphi \|_p = \left( \int_{\mathbb{R}^n} \int_{\mathbb{R}_+}
| \varphi (x,\rho) |^p \omega_m \rho^{m-1} \mathrm{d} \rho \mathrm{d} x \right)^{1/p} \,.
\]
The extension of the gradient square norm, i.e. $\| \nabla \cdot \|_2$, is given by
\[
\| \varphi \|_{\dot{H}^1} := \| \nabla_{x, \rho} \varphi \|_2 = \left( \int_{\mathbb{R}^n} \int_{\mathbb{R}_+} ( | \varphi_\rho |^2 + | \nabla_x \varphi |^2) \omega_m \rho^{m-1} \mathrm{d} \rho
\mathrm{d} x \right)^{1/2} \,,
\]
where the subscript $\rho$ indicates a partial derivative with respect to
$\rho$.  Note that if $m$ is an integer and $\varphi: \mathbb{R}^{m+n} \to \mathbb{C}$ is given $\tilde{\varphi} (x,\tilde{x}) = \varphi (x, |\tilde{x}|)$ for $(x, |\tilde{x}|) \in \mathbb{R}^n \times \mathbb{R}^m$, then
\[
\| \varphi \|_p = \left( \int_{\mathbb{R}^n} \int_{\mathbb{R}^m} | \tilde{\varphi} (x, \tilde{x})|^p \mathrm{d}\tilde{x} \mathrm{d}x \right)^{1/p} \text{ and } \| \varphi \|_{\dot{H}^1} = \left( \int_{\mathbb{R}^n} \int_{\mathbb{R}^m} | \nabla_{x,\tilde{x}} \tilde{\varphi} |^2 (x, \tilde{x}) \mathrm{d}\tilde{x} \mathrm{d}x \right)^{1/2}
\]

Then, the space, $\dot{H}_\mathbb{C}^1 (\mathbb{R}^n \times \mathbb{R}_+, \omega_m \rho^{m-1} \mathrm{d}\rho \mathrm{d}x)$, of complex-valued cylindrically symmetric functions in continuous dimension will be defined as follows: $\varphi \in \dot{H}_\mathbb{C}^1 (\mathbb{R}^n \times \mathbb{R}_+, \omega_m \rho^{m-1} \mathrm{d}\rho \mathrm{d}x)$ if and only if
\begin{enumerate}
\item
$\varphi$ is a complex-valued cylindrically symmetric function with a distributional gradient,
\item
$\| \varphi \|_{\dot{H}^1} < \infty$, and
\item
$\varphi$ is eventually zero in the sense that if $K_\varepsilon = \{ (x, \rho) \in \mathbb{R}^n \times \mathbb{R}_+ \big| | \varphi (x, \rho) | > \varepsilon \}$, then
\[
\int_{K_\varepsilon} \omega_m \rho^{m-1} \mathrm{d}\rho \mathrm{d}x < \infty \,,
\]
for all $\varepsilon > 0$.
\end{enumerate}
The subspace of real-valued functions in $\dot{H}_\mathbb{C}^1 (\mathbb{R}^n \times \mathbb{R}_+, \omega_m \rho^{m-1} \mathrm{d}\rho \mathrm{d}x)$ will be denoted by $\dot{H}^1 (\mathbb{R}^n \times \mathbb{R}_+, \omega_m \rho^{m-1} \mathrm{d}\rho \mathrm{d}x)$. In this setting, we define
\begin{equation}\label{2*aGDef}
2^* := \frac{2(m+n)}{m+n-2} \text{ and } \gamma := \frac{m+n-2}{2} \,.
\end{equation}
Having established this background, we can state Bakry, Gentil, and Ledoux's generalization of the Sobolev Inequality to continuous dimensions with Nguyen's classification of extremals (for reference, see \cite{BaGe} and \cite{Ng}):
\begin{thm}[Sobolev Inequality Extension]\label{SobExtThm}
Let $m+n > 2$, $n$ an integer, $m > 0$ possibly noninteger. Then, for all $\varphi \in \dot{H}_\mathbb{C}^1 ( \mathbb{R}^n \times \mathbb{R}_+, \omega_m \rho^{m-1} \mathrm{d} \rho \mathrm{d}x)$ 
\begin{equation}\label{SobIneqExt}
\| \varphi \|_{2^*} \leq S_{m,n} \| \varphi \|_{\dot{H}^1} \,,
\end{equation}
where
\[
S_{m,n} = \|F_{1,0}\|_{2^*} / \| F_{1,0} \|_{\dot{H}^1} \,,
\]
and
\begin{equation}\label{zFtx0Def}
z F_{s,x_0} (x, \rho) := z s^\gamma (1 + s^2 |x - x_0|^2 + s^2 \rho^2 )^{-\gamma} \,,
\end{equation}
where $x_0 \in \mathbb{R}^n$, $s \in \mathbb{R}_+$, and $z \in \mathbb{C}$. (\ref{SobIneqExt}) gives equality if and only if $\varphi = z F_{s, x_0}$ for some $s > 0$, $x_0 \in \mathbb{R}^n$, and nonzero $z \in \mathbb{C}$.
\end{thm}
Bakry, Gentil, and Ledoux derived (\ref{SobIneqExt}) by relating a Sobolev Inequality on $\mathbb{S}^{n+1}$ to $(\mathbb{R}^n \times \mathbb{R}_+, \omega_m \rho^{m-1} \mathrm{d}\rho \mathrm{d}x)$ via stereographic projection, see p. 322-323 of \cite{BaGe} for detail. Nguyen provides a proof of Theorem \ref{SobExtThm} from a mass-transport approach, and in the process, provides a full classification of extremals for real-valued functions. However, once one has this Sobolev Inequality extension with the classification of extremals for real-valued functions, the generalization to complex-valued functions is easy to deduce; see p. 6 of \cite{Se} for details.

\subsection{Bianchi-Egnell Stability Estimate for Bakry, Gentil, and Ledoux's Generalization of the Sobolev Inequality}

A critical step in Carlen and Figalli's argument for proving Theorem \ref{CaFiStabEst} is the application of a stability estimate for the Sobolev Inequality. This stability estimate, originally proved by Bianchi and Egnell in \cite{BiEg}, is not sufficient to prove Theorem \ref{GN StabEst} using Carlen and Figalli's argument. In order to adapt Carlen and Figalli's argument to prove Theorem \ref{GN StabEst}, we developed a stability estimate for Theorem \ref{SobExtThm}. This stability estimate, proved in \cite{Se}, is an extension to functions on continuous dimensions of Bianchi and Egnell's stability estimate. The proof of this Bianchi-Egnell stability estimate extension  is, at times, quite different from that of Bianchi and Egnell's proof of their original stability estimate of the Sobolev Inequality. The differences are rooted in subtleties stemming from the consideration of functions on continuous dimensions. For example, half of the process of proving these stability estimates is the application of a concentration compactness argument. In the integer dimension case, i.e. the original Bianchi-Egnell stability estimate, this concentration compactness argument was already proved by P.L. Lions in \cite{Lio} and Struwe in \cite{St}. Thus, Bianchi and Egnell applied concentration compactness simply by citing these works. In our extension of Bianchi and Egnell's stability estimate to continuous dimensions, we were not aware of any continuous dimension generalization of the desired concentration compactness. Thus, we had to prove the desired concentration compactness for continuous dimensions. This is just one of several examples of the nontrivial differences that arose from extending the Bianchi-Egnell stability estimate to continuous dimensions. Another notable difference is that we proved a local compactness argument for functions on continuous dimensions, which we used to substitute for the role of the Rellich-Kondrachov Theorem in the concentration compactness argument.

The main theorem that we proved in \cite{Se} is a stability estimate for Theorem \ref{SobExtThm}. The extremals of Theorem \ref{SobExtThm} are given by
\[
z F_{s,x_0} ( x, \rho ) = z s^\gamma ( 1 + s^2 |x -
x_0|^2 + s^2 \rho^2 )^{-\gamma} \,,
\]
for $x_0 \in \mathbb{R}^n$, $s \in \mathbb{R}_+$, and $z \in \mathbb{C} \setminus \{ 0 \}$.  These extremal functions comprise an $(n+3)$-dimensional manifold, $M \subseteq \dot{H}_\mathbb{C}^1 (\mathbb{R}^n \times \mathbb{R}_+, \omega_m \rho^{m-1} \mathrm{d} \rho \mathrm{d}x)$. The distance, $d (\varphi, M)$, between this manifold and a function $\varphi \in \dot{H}_\mathbb{C}^1 (\mathbb{R}^n \times \mathbb{R}_+, \omega_m \rho^{m-1} \mathrm{d} \rho \mathrm{d}x)$ is given by
\begin{equation}\label{Dist}
d (\varphi, M) := \inf_{\psi \in M} \| \varphi - \psi \|_{\dot{H}^1} \,.
\end{equation}
The stability estimate we proved in \cite{Se} is
\begin{thm}[Bianchi-Egnell Extension]\label{BE Ext Thm}
There is a positive constant, $\alpha$, depending only on the parameters, $m$ and $n$, $m>0$ and $n \geq 2$ an integer, so that
\begin{equation}\label{BEIneq}
S_{m,n}^2 \| \varphi \|_{\dot{H}^1}^2 - \| \varphi \|_{2^*}^2 \geq \alpha d (\varphi, M)^2 \,,
\end{equation}
$\forall \varphi \in \dot{H}_\mathbb{C}^1 (\mathbb{R}^n \times \mathbb{R}_+ , \omega_m \rho^{m-1} \mathrm{d} \rho \mathrm{d}x)$. Furthermore, the result is sharp in the sense that it is no longer true if $d (\varphi, M)^2$ in (\ref{BEIneq}) is replaced with $d (\varphi, M)^\beta \| \varphi \|_{\dot{H}^1}^{2 - \beta}$, where $\beta < 2$.
\end{thm}

\section{From Bakry, Gentil, and Ledoux's Generalization of the Sobolev Inequality to Del Pino and Dolbeault's Family of Gagliardo-Nirenberg Inequalities}

The heart of the proof of Theorem \ref{GN StabEst} is an argument connecting the one parameter family of GN inequalities given by (\ref{GN ineq's}) to the extension of the Sobolev Inequality to continuous dimensions proved by Bakry, Gentil, and Ledoux in \cite{BaGe}. This connection, which is a generalization of Carlen and Figalli's Proposition \ref{CaFiDeficits}, is summarized in the following
\begin{prop}\label{Deficits}
Let $1 < t \leq \frac{n}{n-2}$ if $n >2$ or $t > 1$ if $n = 2$. Also, let $u \in \dot{H}^1 (\mathbb{R}^n)$ be a nonnegative function such that
\begin{equation}\label{Normalization}
\| u \|_{t+1}^{t+1} / \| \nabla u \|_2^2 = \| v \|_{t+1}^{t+1} / \| \nabla v \|_2^2 \,.
\end{equation}
Let $\varphi_u: \mathbb{R}_+ \times \mathbb{R}^n \to \mathbb{R}$ be given by
\begin{equation}\label{varphi_u}
\varphi_u (\rho, x) = [w_u (x) + \rho^2]^{-\frac{m+n-2}{2}} \,,
\end{equation}
where
\begin{equation}\label{w_u Def}
w_u (x) = u^{-\frac{4}{m + 2n - 4}} (x) \,.
\end{equation}
Then
\begin{equation}\label{DeficitEq}
C_1^{-1} (S_{m,n}^2 \| \nabla \varphi_u \|_2^2 - \| \varphi_u \|_{2^*}^2) = A_{n,t}^{4t/2^*} \| \nabla u \|_2^{\mu 4t/2^*}\| u \|_{t + 1}^{(1 - \mu)4t / 2^*} - \| u \|_{2t}^{4t / 2^*} = \delta_{GN}[u] \,,
\end{equation}
where
\begin{align}
t =& \frac{m + 2n}{m + 2n - 4} \text{ , and} \label{t Def} \\
C_1 =& \left( \int_{\mathbb{R}_+} [1 + \theta^2]^{-(m+n)} \omega_m \theta^{m-1} \mathrm{d} \theta \right)^{2/2^*} \,. \label{C_1 Def}
\end{align}
\end{prop}

\begin{proof}
Before proceeding, we establish some notation. Define $\mathrm{d} \Omega (\rho)$ as
\begin{equation}\label{OmegaDef}
\mathrm{d} \Omega (\rho) := \omega_m \rho^{m-1} \mathrm{d}\rho \,.
\end{equation}
Through explicit calculation and making the change of variables given by $\theta = w_u^{1/2}(y) \rho$, we have that
\begin{align}\label{GrSqNormCalc}
& \gamma^{-2} \| \varphi_u \|_{\dot{H}^1}^2 \nonumber \\
=& \int_{\mathbb{R}^n} | \nabla_x w_u |^2 w_u ^{-\frac{m + 2n}{2}} \mathrm{d}x \int_{\mathbb{R}_+} [1 + \theta^2 ]^{-(m+n)} \mathrm{d} \Omega (\theta) + \int_{\mathbb{R}^n} 4 w_u^{-\frac{m + 2n - 2}{2}} \mathrm{d}x \int_{\mathbb{R}_+} \theta^2 [1 + \theta^2]^{-(m+n)} \mathrm{d} \Omega (\theta) \nonumber \\
&\text{which by (\ref{w_u Def})} \nonumber \\
=& \left( \frac{4}{m+2n-4} \right)^2 \int_{\mathbb{R}^n} | \nabla_x u |^2 \mathrm{d}x \int_{\mathbb{R}_+} [1 + \theta^2 ]^{-(m+n)} \omega_m \mathrm{d} \Omega (\theta) + 4 \int_{\mathbb{R}^n} u^{t+1} \mathrm{d}x \int_{\mathbb{R}_+} \theta^2 [1 + \theta]^{-(m+n)} \mathrm{d} \Omega (\theta) \,,
\end{align}
and
\begin{align}\label{2*NormCalc}
\| \varphi_u \|_{2^*}^2 =& \left( \int_{\mathbb{R}^n} w_u^{-(m + 2n)/2} \mathrm{d}x \int_{\mathbb{R}_+} [1 + \theta^2]^{-(m+n)} \mathrm{d}\Omega(\theta) \right)^{2/2^*} \,, \text{ which by (\ref{w_u Def})} \nonumber \\
=& \left( \int_{\mathbb{R}^n} u^{2t} \mathrm{d}x \int_{\mathbb{R}_+} [1 + \theta^2]^{-(m+n)} \mathrm{d}\Omega(\theta) \right)^{2/2^*} \,.
\end{align}
We will use (\ref{GrSqNormCalc}) and (\ref{2*NormCalc}) to derive the GNS deficit with $\left( \| \nabla u \|_2^\mu \| u \|_{t + 1}^{(1 - \mu)} \right)^{4t/2^*}$ coming from (\ref{GrSqNormCalc}) and $\| u \|_{2t}^{4t/2^*}$ coming from (\ref{2*NormCalc}) as per (\ref{Normalization}).

Combining (\ref{GrSqNormCalc}) and (\ref{2*NormCalc}), we have that
\begin{equation}\label{PartEq}
S_{m,n}^2 \| \varphi_u \|_{\dot{H}^1}^2 - \| \varphi_u \|_{2^*}^2 = C_1 \left( C_2 \| \nabla u \|_2^2 + C_3 \| u \|_{t+1}^{t+1} - \| u \|_{2t}^{4t/2^*} \right) \,,
\end{equation}
where
\begin{eqnarray}
C_1 &=& \left( \int_{\mathbb{R}_+} [1 + \theta^2]^{-(m+n)} \mathrm{d} \Omega (\theta) \right)^{2/2^*} \nonumber \\
C_2 &=& C_1^{-1} S_{m,n}^2 \gamma^2 \left[ \frac{4}{m + 2n - 4} \right]^2  \int_{\mathbb{R}_+} [1 + \theta^2 ]^{-(m+n)} \mathrm{d} \Omega (\theta) \nonumber \\
C_3 &=&  C_1^{-1} S_{m,n}^2 \gamma^2 \left[ \frac{4}{m + 2n - 4} \right]^2 4 \int_{\mathbb{R}_+} \theta^2 [1 + \theta]^{-(m+n)} \mathrm{d} \Omega (\theta) \,. \nonumber
\end{eqnarray}
If we take $u = v$, then $\varphi_u = F_{1,0}$, and (\ref{PartEq}) gives
\begin{eqnarray}
C_1 \left( C_2 \| \nabla v \|_2^2 + C_3 \| v \|_{t+1}^{t+1} - \| v \|_{2t}^{4t/2^*} \right) &=& S_{m,n}^2 \| F_{1,0} \|_{\dot{H}^1}^2 - \| F_{1,0} \|_{2^*}^2 \nonumber \\
&=& 0 \nonumber \\
&=& C_1 \left( A_{n,t}^{4t/2^*} \| \nabla v \|_2^{\mu 4t/2^*} \| v \|_{t+1}^{(1-\mu)4t/2^*} - \| v \|_{2t}^{4t/2^*} \right) \,. \nonumber
\end{eqnarray}
Thus,
\begin{equation}\label{vNormEq}
C_2 \| \nabla v \|_2^2 + C_3 \| v \|_{t+1}^{t+1} = A_{n,t}^{4t/2^*} \| \nabla v \|_2^{\mu 4t/2^*} \| v \|_{t+1}^{(1-\mu)4t/2^*} \,.
\end{equation}
In fact we claim that (\ref{Normalization}) implies that
\begin{equation}\label{GoodEq1}
C_2 \| \nabla u \|_2^2 + C_3 \| u \|_{t+1}^{t+1} = A_{n,t}^{4t/2^*} \| \nabla u \|_2^{\mu 4 t/2^*} \| u \|_{t+1}^{(1-\mu)4t/2^*} \,.
\end{equation}
We begin verifying (\ref{GoodEq1}) by observing that it is equivalent to
\begin{equation}\label{GoodEq2}
C_2 + C_3 \frac{\| u \|_{t+1}^{t+1}}{\| \nabla u \|_2^2} = A_{n,t}^{4t/2^*} \frac{(\| u \|_{t+1}^{t+1})^{(1-\mu)4t/2^*(t+1)} }{ (\| \nabla u \|_2^2)^{1 - \mu 2t/2^*} } \,,
\end{equation}
which must hold if
\begin{equation}\label{GoodEq3}
\frac{(1 - \mu)4t}{2^*(t+1)} = 1 - \frac{\mu 2t}{2^*} \,,
\end{equation}
because (\ref{GoodEq2}) holds for $u = v$ as per (\ref{vNormEq}), and (\ref{GoodEq3}) would imply that the left hand side and the right hand side is the same for all $u$ obeying (\ref{Normalization}). A direct calculation using the formulas for $\mu$, $2^*$, and $t$ given by (\ref{GN ineq's}), (\ref{2*aGDef}), (\ref{t Def}) respectively verifies (\ref{GoodEq3}). Thus, (\ref{GoodEq1}) holds for $u$ obeying (\ref{Normalization}). Combining this with (\ref{PartEq}) we conclude Proposition \ref{Deficits}.
\end{proof}

\section{Controlling the Infimum in the Bianchi-Egnell Theorem}

The family of functions
\[
z F_{s, x_0} (\rho, x) = z s^\gamma ( 1 + s^2 \rho^2 + s^2 |x - x_0|^2 )^{-\gamma} \,, \text{ for } z \in \mathbb{C} \setminus \{ 0 \}, \text{ } s \in \mathbb{R}_+, \text{ and } x_0 \in \mathbb{R}^n \,,
\]
comprises all the optimizers of the Sobolev Inequality.  For convenience, we define a function, $F$, by
\[
F (\rho, x) := F_{1,0} (\rho, x) \,.
\]

The Bianchi-Egnell extension stability result from \cite{Se} combined with the Sobolev Inequality extension from \cite{BaGe} asserts the existence of a constant $C_0 := C_0 (m,n)$, depending on the dimensions $m$ and $n$, such that
\begin{equation}
C_0 S_{m,n}^2 \left( S_{m,n}^2 \| \varphi \|_{\dot{H}^1}^2 - \| \varphi \|_{2^*}^2 \right) \geq \inf_{z, s, x_0} \| \varphi_u - z F_{s, x_0} \|_{2^*}^2 \,.
\end{equation}
Hence, whenever $u$ satisfies (\ref{Normalization}) and $\varphi_u$ is given by $(\ref{varphi_u})$,
\begin{equation}
C_0 C_1 S_{m,n}^2 \delta_{GN} [u] \geq \inf_{z, s, x_0} \| \varphi_u - z F_{s, x_0} \|_{2^*}^2 \,.
\end{equation}
Let us observe that the renormalization $\| u \|_{2t} = \| v \|_{2t}$ is equivalent to $\| \varphi_u \|_{2^*} = \| F \|_{2^*}$.

\begin{lm}\label{ControlInfLm}
Let $\varphi$ be given by $\varphi(\rho, x) = [ G(x) + \rho^2 ]^{- \gamma}$ with $G: \mathbb{R}^n \to \mathbb{R}$ nonnegative.  Suppose that $\| \varphi \|_{2^*} = \| F \|_{2^*}$.  Then there exists a universal constant $C_4$ such that, for all $\delta > 0$ with
\begin{equation}\label{DeltaConstraint}
\delta^{1/2} < \min \left\{ \| F \|_{2^*}^{2^*} \,, \frac{4^\gamma - 3^\gamma}{12^\gamma \cdot 3^{1 + 1/2^*}} \left( \frac{\omega_m \omega_n}{mn} \right)^{1/2^*} \,, \frac{\gamma}{12^\gamma \cdot 2^{4 + \gamma + 1/2^*} } \left( \frac{\omega_m \omega_n}{mn} \right)^{1/2^*} \right\} \,,
\end{equation}
whenever
\[
\| \varphi - z F_{s,x_0} \|_{2^*} \leq \delta^{1/2} \text{, for some } z, s, x_0 \,,
\]
then
\[
\| \varphi - F_{1,x_0} \|_{2^*} \leq C_4 \delta^{1/2} \,.
\]
\end{lm}
\noindent As can be seen from the proof, a possible choice for $C_4$ is $ 2 + 12^\gamma \cdot 2^{3+\gamma+1/2^*} \left( \frac{mn}{\omega_m \omega_n} \right)^{1/2^*} \| F \|_{2^*}$.

\begin{proof}
The proof to this lemma should follow in two steps.

Let $\varphi$ and $z F_{s,x_0}$ be such that $\| \varphi - z F_{s,x_0} \|_{2^*} < \delta^{1/2}$.

\textit{Step 1: we show that we can assume that $z = 1$.}  First of all, observe that $z \geq 0$, or else

\begin{eqnarray}
 \delta^{2^*/2} &\geq& \int_{\mathbb{R}^n} \int_{\mathbb{R}_+} | \varphi - z F_{s,x_0}|^{2^*} \mathrm{d} \Omega (\rho) \mathrm{d}x \nonumber \\
&\geq& \int_{\mathbb{R}^n} \int_{\mathbb{R}_+} |\varphi|^{2^*} + |z F_{s,x_0}|^{2^*} \mathrm{d}\Omega(\rho) \mathrm{d}x \nonumber \\
&\geq& \| \varphi \|_{2^*}^{2^*} = \| F \|_{2^*}^{2^*} \,, \nonumber
\end{eqnarray}
contradicting (\ref{DeltaConstraint}).

Now, for any $z,s > 0$, $\| z F_{s,x_0} \|_{2^*} = z \| F \|_{2^*} = z \| \varphi \|_{2^*}$.  Thus,
\begin{eqnarray}
 |z-1| \| F \|_{2^*} &=& \left| \| z F_{s,x_0} \|_{2^*} - \| \varphi \|_{2^*} \right| \nonumber \\
&\leq& \| \varphi - z F_{s,x_0} \|_{2^*} \nonumber \\
&<& \delta^{1/2} \,. \nonumber
\end{eqnarray}
Employing the triangle identity
\begin{align}\label{Est1}
 \| \varphi - F_{s,x_0}\|_{2^*} \leq& \| \varphi - z F_{s,x_0} \|_{2^*} + \| z F_{s,x_0} - F_{s,x_0} \|_{2^*} \nonumber \\
=& \| \varphi - z F_{s,x_0} \|_{2^*} + |z-1| \| F \|_{2^*} \nonumber \\
<& 2 \delta^{1/2}
\end{align}
Thus, up to enlarging the constant, we may replace $z$ by 1.

\textit{Step 2: we can assume that $s = 1$.}

\noindent Note that by a change of scale, we can rewrite (\ref{Est1}) as
\begin{equation}\label{Est2}
\| [s G(x/s) + \rho^2/s]^{-\gamma} - [1 + |x - s x_0|^2 + \rho^2]^{-\gamma} \|_{2^*} \leq 2 \delta^{1/2}
\end{equation}
Let $A := \left\{ (\rho, x) \in [0, \infty) \times \mathbb{R}^n \big| \rho, |x-x_0| \leq 1 \right\}$.  By a simple Fubini argument, we know that any set $B \subseteq A$ with measure greater than
\[
\int_{|x-x_0| \leq 1} \int_{1/4}^1 \omega_m \rho^{m-1} \mathrm{d} \rho \mathrm{d}x = \frac{\omega_m \omega_n (4^{m} - 1)}{mn 4^m}
\]
there exists $\bar{x} \in \left\{ x \in \mathbb{R}^n \big| |x - s x_0| \leq 1 \right\}$ such that $B \cap \left( [0, \infty) \times \{ \bar{x} \} \right)$ must intersect both
\begin{equation}\label{Intersection}
A \cap \{ (\rho, \bar{x}) \in [0, \infty) \times \mathbb{R}^n | \rho \leq 1/4 \} \text{ and } A \cap \{ (\rho, \bar{x}) \in [0, \infty) \times \mathbb{R}^n | \rho \geq 3/4 \}
\end{equation}
(If this were not the case, the Fubini Theorem would imply that the measure of $B$ would be smaller than $\omega_m \omega_n (4^{m} - 1)/ mn 4^{m}$.)

Now, applying Chebyshev's Inequality, by (\ref{Est2}), we get the existence of a set $B \subseteq A$ of measure at least $\omega_m \omega_n (2 \cdot 4^{m} - 1)/ mn 2 \cdot 4^m$ such that
\begin{equation}\label{Est3}
\left| [ sG(x/s) + \rho^2/s]^{-\gamma} - [1 + |x - sx_0|^2 + \rho^2]^{-\gamma} \right| \leq \delta^{1/2} 2^{1 + 1/2^*} \left( \frac{mn}{\omega_m \omega_n} \right)^{1/2^*} \,, \text{ } \forall (\rho,x) \in B \,.
\end{equation}
Set $\xi (\rho, x) := 1 + |x - s x_0|^2 + \rho^2$ and $\eta (\rho, x) := s G(x/s) + \rho^2/s$, so that (\ref{Est3}) becomes 
\begin{equation}\label{Est4}
\left| \eta^{-\gamma} - \xi^{-\gamma} \right| \leq \delta^{1/2} 2^{1 + 1/2^*} \left( \frac{mn}{\omega_m \omega_n} \right)^{1/2^*} \text{, inside } B
\end{equation}
We observe that $\xi \leq 3$ in $B$.  Moreover,
\begin{eqnarray}
 \eta^{-\gamma} &\geq& \xi^{-\gamma} - \left| \eta^{-\gamma} - \xi^{-\gamma} \right| \nonumber \\
&\geq& 3^{-\gamma} - \delta^{1/2} 2^{1+1/2^*} \left( \frac{mn}{\omega_m \omega_n} \right)^{1/2^*} \nonumber \\
&\geq& 4^{-\gamma} \text{ , by (\ref{DeltaConstraint}),} \nonumber
\end{eqnarray}
that is
\[
\eta \leq 4 \text{, inside $B$.}
\]
Combining this with (\ref{Est4}) we conclude that
\begin{equation}\label{Est5}
| \xi^\gamma - \eta^\gamma | \leq 12^\gamma \cdot 2^{1+1/2^*} \left( \frac{mn}{\omega_m \omega_n} \right)^{1/2^*} \delta^{1/2} \text{, inside $B$.}
\end{equation}
Note that $\xi = 1 + |x - s x_0|^2 + \rho^2 \geq 1$.  Thus, by (\ref{Est5}),
\[
1 - \eta^\gamma \leq 12^\gamma \cdot 2^{1+1/2^*} \left( \frac{mn}{\omega_m \omega_n} \right)^{1/2^*} \delta^{1/2} \text{, inside $B$.}
\]
Rearranging terms in the above and using the assumption that $\delta^{1/2} \leq 12^{-\gamma} \cdot 2^{-1 - 1/2^*} \left( \frac{\omega_m \omega_n}{mn} \right)^{1/2^*}$, we conclude that
\[
\eta \geq 1 / 2 \text{, inside $B$.} \nonumber
\]
Applying the Mean Value Theorem
\begin{align}
| \xi^\gamma - \eta^\gamma | &= \gamma | \lambda |^{\gamma-1} | \eta- \xi | \text{, for some $\lambda$ between $\eta$ and $\xi$} \nonumber \\
&\geq \gamma 2^{-\gamma+1} | \eta - \xi | \text{, inside $B$.} \nonumber
\end{align}
Combining this with (\ref{Est4}), we obtain
\[
| \eta - \xi | \leq \frac{12^\gamma \cdot 2^{\gamma+1/2^*}}{\gamma} \left( \frac{mn}{\omega_m \omega_n} \right)^{1/2^*} \delta^{1/2} \,, \text{ inside $B$.}
\]
Or equivalently,
\begin{equation}\label{Est6}
\left| 1 + |x-sx_0|^2 + \rho^2 - s G(x/s) - \rho^2/s \right| \leq \frac{12^\gamma \cdot 2^{\gamma+1/2^*}}{\gamma} \left( \frac{mn}{\omega_m \omega_n} \right)^{1/2^*} \delta^{1/2} \,, \text{ } \forall (x,\rho) \in B \,.
\end{equation}
By the observation in the line before (\ref{Intersection}) and the line (\ref{Intersection}), we know that we have chosen $B$ large enough that there exists $\bar{x} \in \mathbb{R}^n$ with 
\begin{equation}\label{inB}
(\rho_1, \bar{x}), (\rho_2, \bar{x}) \in B \,,
\end{equation}
for $\rho_1 \in [0, \frac{1}{4}]$ and $\rho_2 \in [\frac{3}{4}, 1]$.  Then the above estimate, (\ref{Est6}), gives
\begin{align}\label{Est7}
\frac{1}{2} \left| 1 - \frac{1}{s} \right| &\leq (\rho_2^2 - \rho_1^2) \left| 1 - \frac{1}{s} \right| \nonumber \\
&\leq \left| 1 + |\bar{x} - s x_0|^2 + \rho_2^2 - s G(\bar{x}/s) - \rho_2^2 / s \right| + \left| 1 + |\bar{x} - s x_0|^2 + \rho_1^2 - s G(\bar{x}/s) - \rho_1^2 / s \right| \nonumber \\
&\leq \frac{12^\gamma \cdot 2^{1+\gamma+1/2^*}}{\gamma} \left( \frac{mn}{\omega_m \omega_n} \right)^{1/2^*} \delta^{1/2} \text{, by (\ref{Est6}) and (\ref{inB}).}
\end{align}
Using (\ref{DeltaConstraint}) and the identity $(s - 1)(1 - (1 - 1/s)) = (1 - 1/s)$, we deduce
\begin{equation}\label{Est8}
| s - 1 | \leq  \frac{12^\gamma \cdot 2^{2+\gamma+1/2^*}}{\gamma} \left( \frac{mn}{\omega_m \omega_n} \right)^{1/2^*} \delta^{1/2} \,.
\end{equation}
Note that the above and (\ref{DeltaConstraint}) imply that
\begin{equation}\label{Est9}
|s-1| \leq 1/2 \,.
\end{equation}
Next we observe that
\begin{align}
 | \partial_s F_{s, x_0} (x, r) | &= \left| \frac{\gamma}{s} \cdot \frac{1 - s^2 \rho^2 - s^2 |x - x_0|^2}{1 + s^2 \rho^2 + s^2 |x-x_0|^2} F_{s,x_0}(\rho,x) \right| \nonumber \\
&\leq \frac{\gamma}{s} |F_{s,x_0}(\rho,x)| \nonumber \\
&\leq 2 \gamma|F_{s,x_0}(\rho,r)| \text{, } \forall s \in \left[ 1 / 2, 3 / 2 \right] \,. \nonumber
\end{align}
Combining the above with Minkowski's Integral Inequality, we deduce the following:
\begin{equation}\label{Est10}
 \| F_{s,x_0} - F_{1,x_0} \|_{2^*} \leq 2 \gamma | s - 1 | \| F \|_{2^*} \text{, } \forall s \in \left[ 1 / 2, 3 / 2 \right] \,.
\end{equation}
Combining (\ref{Est10}) with (\ref{Est1}), (\ref{Est8}), and (\ref{Est9}) we finally obtain
\begin{equation}
 \| \varphi - F_{1,x_0} \|_{2^*} \leq \left[ 2 + 12^\gamma \cdot 2^{3+\gamma+1/2^*} \left( \frac{mn}{\omega_m \omega_n} \right)^{1/2^*} \| F \|_{2^*} \right] \delta^{1/2} \,.
\end{equation}
\end{proof}

\section{Bounding $\|u^{2t}-v^{2t}\|_1$}

\begin{lm}\label{BoundingLm}
 Let $u \in \dot{H}^1 (\mathbb{R}^n)$ be a nonnegative function satisfying (\ref{Normalization}), and let $\varphi_u$ be defined as in Proposition \ref{Deficits}.  Suppose that $\| \varphi_u - F_{1,x_0} \|_{2^*} \leq 1$.  Then
\begin{equation}\label{Thing}
\| u^{2t} - v^{2t} (\cdot - x_0) \|_1 \leq C_5 \| \varphi_u - F_{1,x_0} \|_{2^*} \,,
\end{equation}
for some constant $C_5 := C_5 (m,n)$, depending upon the dimensions $m$ and $n$.
\end{lm}

As can be seen from the proof, a possible choice for $C_5$ is
\[
C_5 (m,n) =
\begin{cases}
$$
4 \max \left\{ \frac{2^\frac{m+2n}{2} (m+2n)}{m+n-2} \| v \|_{2t}^{2t(1-1/2^*)} \left( \int_{\mathbb{R}_+} \frac{\omega_m \theta^{m-1}}{(1+\theta^2)^{2^* (m+n)/2}} \mathrm{d}\theta \right)^{-1/2^*} \,, \frac{2^{3 \cdot 2^*}}{(m+n-2)^{2^*}} \left( \int_0^{\sqrt{2}} \frac{\omega_m \tilde{\theta}^{m-1}}{(1+\tilde{\theta}^2)^{m+n}} \mathrm{d} \tilde{\theta} \right)^{-1} \right\} \\ 
\text{for $2 < m+n \leq 4$} \\
4 \max \left\{ \frac{2^\frac{m+2n}{2}(m+2n)}{m+n-2} \| v \|_{2t}^{2t (1-1/2^*)} \left( \int_{\mathbb{R}_+} \frac{\omega_m \theta^{m-1}}{(1+\theta^2)^{2^* (m+n)/2}} \mathrm{d}\theta \right)^{-1/2^*} \,, \frac{2^{2^* (m+n-1)}}{(m+n-2)^{2^*}} \left( \int_0^{\sqrt{2}} \frac{\omega_m \tilde{\theta}^{m-1}}{(1+\tilde{\theta^2})^{m+n}} \mathrm{d}\tilde{\theta} \right)^{-1} \right\} \\
\text{for $m+n \geq 4$}
$$
\end{cases}
\]
We also remark that by considering $u$ of the form $v + \varepsilon \phi$ with $\varepsilon > 0$ small, one sees that the unit in the above estimate is optimal.

\begin{proof}
We will provide a sketch of the proof. The results of some precise calculations will be used here, while the actual steps of the calculations will be presented in the appendix. The steps of these calculations are tedious and are generalizations of the same steps for the proof of Carlen and Figalli's analogue of Lemma \ref{BoundingLm}, Lemma 2.4 in \cite{CaFi}. Before proceeding, we introduce the following notation:
\[
 H(x) := 1 + |x-x_0|^2 \,.
\]
We begin by observing that
\begin{align}\label{Est11}
 \| \varphi_u - F_{1,x_0} \|_{2^*}^{2^*} =& \int_{\mathbb{R}^n} \left( \int_{\mathbb{R}_+} \frac{ \left| (w_u + \rho^2)^\gamma - (H + \rho^2)^\gamma \right|^{2^*}}{(w_u +\rho^2)^{m+n} (H+\rho^2)^{m+n}} \omega_m \rho^{m-1} \mathrm{d} \rho \right) \mathrm{d}x \nonumber \\
=& \int_{\{ w_u < H \}} \left( \int_{\mathbb{R}_+} \frac{ \left| (w_u + \rho^2)^\gamma - (H + \rho^2)^\gamma \right|^{2^*}}{(w_u +\rho^2)^{m+n} (H+\rho^2)^{m+n}} \omega_m \rho^{m-1} \mathrm{d} \rho \right) \mathrm{d}x \nonumber \\
&+ \int_{\{ w_u > H \}} \left( \int_{\mathbb{R}_+} \frac{ \left| (w_u + \rho^2)^\gamma - (H + \rho^2)^\gamma \right|^{2^*}}{(w_u +\rho^2)^{m+n} (H+\rho^2)^{m+n}} \omega_m \rho^{m-1} \mathrm{d} \rho \right) \mathrm{d}x
\end{align}
First, we will estimate the first integral term in the right hand side of the above. Then, we will use symmetry to obtain a similar estimate for the second integral in the right hand side of the above.  We split $\{w_u < H \}$ as follows, $\{w_u < H \} = \{ H/2 \leq w_u < H \} \cup \{ w_u < H/2 \} =: A_1 \cup A_2$.

On $A_1$, we use the Mean Value Theorem and a change of variables to compute that
\begin{equation}\label{Est12}
\int_{A_1} \left( \int_{\mathbb{R}_+} \frac{|(w_u + \rho^2)^\gamma - (H+\rho^2)^\gamma|^{2^*}}{(w_u + \rho^2)^{m+n} (H + \rho^2)^{m+n}} \mathrm{d} \Omega(\rho) \right) \mathrm{d}x \geq \gamma^{2^*} \int_{\mathbb{R}_+} (1+\theta^2)^{-(2^* + m + n)} \mathrm{d}\Omega (\theta) \int_{A_1} \frac{|H-w_u|^{2^*}}{H^{2^* + \frac{m+2n}{2}}} \mathrm{d}x \,.
\end{equation}
Also, pointwise on $A_1$, we use the Mean Value Theorem to deduce that
\begin{equation}\label{Est13}
 |u^{2t} - v^{2t}| = \left| w_u^{-\frac{m+2n}{2}} - H^{-\frac{m+2n}{2}} \right| \leq 2^\frac{m+2n}{2} (m+2n) \frac{|H-w_u|}{H^{\frac{m+2n}{2}+1}} \,.
\end{equation}
Note that
\begin{equation}\label{Eq1}
\frac{m+2n}{2} = \frac{m + 2n}{2 \cdot 2^*} + \frac{(m+2n)(2^*-1)}{2 \cdot 2^*} \,.
\end{equation}
Combining (\ref{Est13}), (\ref{Eq1}), and using Holder's Inequality we calculate that
\begin{align}\label{Est14}
 &\int_{A_1} |u^{2t}-v^{2t}| \mathrm{d}x \nonumber \\
\leq& \frac{2^\frac{m+2n}{2} (m+2n)}{\gamma} \left( \| v \|_{2t}^{2t} \right)^{1-1/2^*} \left( \int_{\mathbb{R}_+} (1+\theta^2)^{-(2^* + m + n)} \mathrm{d} \Omega (\theta) \right)^{-1/2^*} \cdot \nonumber \\
 & \left[ \int_{A_1} \left( \int_{\mathbb{R}_+} \frac{|(w_u + \rho^2)^\gamma - (H +\rho^2)^\gamma|^{2^*}}{(w_u + \rho^2)^{m+n} (H + \rho^2)^{m+n}} \mathrm{d} \Omega (\theta) \right) \mathrm{d}x \right]^{1/2^*} \,,
\end{align}
where the last inequality is deduced from (\ref{Est12}).

On $A_2$, we use the Mean Value Theorem to compute that
\begin{align}\label{Est15}
& \int_{A_2} \left( \int_{\mathbb{R}_+}  \frac{|(H+\rho^2)^\gamma - (w_u+\rho^2)^\gamma|^{2^*}}{|w_u+\rho^2|^{m+n} |H+\rho^2|^{m+n}} \mathrm{d} \Omega (\rho) \right) \mathrm{d}x \nonumber \\
\geq& \left( \frac{\gamma}{2} \right)^{2^*} \int_{A_2} H^{2^*} \left( \int_{\mathbb{R}_+} \frac{(K+\rho^2)^{m+n-{2^*}} }{|w_u+\rho^2|^{m+n} |H + \rho^2|^{m+n}} \mathrm{d} \Omega (\rho) \right) \mathrm{d}x \text{, for some $H / 2 \leq K \leq H$.}
\end{align}
Before proceeding, it is worth noting that in the work of E. A. Carlen and A. Figalli, \cite{CaFi}, they prove an analogue of Lemma \ref{BoundingLm} for which the proof is a bit more straightforward. The main reason for this is that in their setting the dimensions of the relevant functions lead to favorable cancellations. They prove an analogue of Lemma \ref{BoundingLm} that is almost the same as proving Lemma \ref{BoundingLm} with $m=n=2$. Carrying out some of the steps of proving Lemma \ref{BoundingLm} for $m=n=2$ illustrates the aspects that make their proof more straightforward. If we set $m=n=2$, the numerator of first line of (\ref{Est15}) simplifies to $|H-w_u|^4$.  This makes calculating a good lower bound considerably more simple.  Indeed, if one tries to prove Lemma \ref{BoundingLm} for $m=n=2$ by deriving (\ref{Est15}), we have that $(K+\rho^2)^{m+n-{2^*}} = (K+\rho^2)^{2+2-4} = 1$ and we need not worry about bounding the numerator in the inner integrand of (\ref{Est15}).

We will treat (\ref{Est15}) in two cases.  

\textbf{Case 1: $2 < m+n \leq 4$.}  In this case, $m+n-2^* \leq 0$. Thus,
\begin{align}\label{Est16}
& \int_{A_2} H^{2^*} \left( \int_{\mathbb{R}_+} \frac{(K+\rho^2)^{m+n-{2^*}}}{|w_u+\rho^2|^{m+n} |H + \rho^2|^{m+n}} \mathrm{d} \Omega (\rho) \right) \mathrm{d}x \nonumber \\
\geq& \int_{A_2} H^{2^*} \left( \int_{\mathbb{R}_+} \frac{(H+\rho^2)^{m+n-2^*}}{|w_u+\rho^2|^{m+n} |H + \rho^2|^{m+n}} \mathrm{d} \Omega (\rho) \right) \mathrm{d}x\text{, changing variables twice and changing bounds} \nonumber \\
\geq& 2^{-2^*} \int_0^{\sqrt{2}} (1+\tilde{\theta}^2)^{-(m+n)} \mathrm{d} \Omega (\tilde{\theta})  \int_{A_2} w_u^{-\frac{m+2n}{2}} - H^{-\frac{m+2n}{2}} \mathrm{d}x \nonumber \\
\geq& 2^{-2^*} \int_0^{\sqrt{2}} (1+\tilde{\theta}^2)^{-(m+n)} \mathrm{d} \Omega (\tilde{\theta})  \int_{A_2} |u^{2t} - v^{2t}| \mathrm{d}x \,.
\end{align} 

\textbf{Case 2: $m+n \geq 4$.} In this case, $m+n-2^* \geq 0$.  Thus,
\begin{align}\label{Est17}
 & \int_{A_2} H^{2^*} \left( \int_{\mathbb{R}_+} \frac{(K+\rho^2)^{m+n-{2^*}}}{|w_u+\rho^2|^{m+n} |H + \rho^2|^{m+n}} \mathrm{d} \Omega (\rho) \right) \mathrm{d}x \nonumber \\
 \geq& \int_{A_2} H^{2^*} \left( \int_{\mathbb{R}_+} \frac{(\frac{H}{2}+\rho^2)^{m+n-2^*}}{|w_u+\rho^2|^{m+n} |H + \rho^2|^{m+n}} \mathrm{d} \Omega (\rho) \right) \mathrm{d}x \text{, changing variables twice and changing bounds} \nonumber \\
 \geq& 2^{-2(m+n) + 2^*} \int_0^{\sqrt{2}} (1+\tilde{\theta^2})^{-(m+n)} \mathrm{d} \Omega (\tilde{\theta}) \int_{A_2} w_u^{-\frac{m+2n}{2}} \mathrm{d}x \nonumber \\
 \geq& 2^{-2(m+n) + 2^*} \int_0^{\sqrt{2}} (1+\tilde{\theta^2})^{-(m+n)} \mathrm{d}\Omega (\tilde{\theta}) \int_{A_2} |u^{2t} - v^{2t}| \mathrm{d}x \,.
\end{align}

Combining (\ref{Est15}) with (\ref{Est16}) and (\ref{Est17}), we get
\begin{align}
& \int_{A_2} \left( \int_{\mathbb{R}_+}  \frac{|(H+\rho^2)^\gamma - (w_u+\rho^2)^\gamma|^{2^*}}{|w_u+\rho^2|^{m+n} |H+\rho^2|^{m+n}} \mathrm{d} \Omega (\rho) \right) \mathrm{d}x \nonumber \\
\geq \nonumber
&\begin{cases}
$$
\left( \frac{\gamma}{4} \right)^{2^*} \int_0^{\sqrt{2}} (1+\tilde{\theta}^2)^{-(m+n)} \mathrm{d} \Omega (\tilde{\theta})  \int_{A_2} |u^{2t} - v^{2t}| \mathrm{d}x \text{, for $2<m+n\leq4$} \nonumber \\
\frac{\gamma^{2^*}}{2^{2(m+n)}} \int_0^{\sqrt{2}} (1+\tilde{\theta^2})^{-(m+n)} \mathrm{d} \Omega (\tilde{\theta}) \int_{A_2} |u^{2t} - v^{2t}| \mathrm{d}x \text{, for $m+n \geq 4$.}
$$
\end{cases}
\end{align}
Combining this with (\ref{Est11}),
\begin{equation}\label{Est18}
\frac{C_5}{2} \| \varphi_u - F_{1,x_0} \|_{2^*}^{2^*} \geq \frac{C_5}{2} \int_{A_2} \left( \int_{\mathbb{R}_+}  \frac{|(H+\rho^2)^\gamma - (w_u+\rho^2)^\gamma|^{2^*}}{|w_u+\rho^2|^{m+n} |H+\rho^2|^{m+n}} \mathrm{d} \Omega (\rho) \right) \mathrm{d}x \geq \int_{\{u>v\}} |u^{2t} - v^{2t}| \mathrm{d}x \,.
\end{equation}
Using the same arguments as above but switching the roles of $w_u$ and $H$, we deduce the analogue of (\ref{Est18}) for the integral on the right hand side on $\{ u < v \}$ instead of $\{ u > v \}$. Combining this with (\ref{Est18}) and the assumption that $\| \varphi_u - F_{1,x_0} \|_{2^*} \leq 1$, we conclude (\ref{Thing}).
\end{proof}

\section{Proof of Theorem 1.1}

In the following, it will be useful to recall that
\[
u_\lambda (y) = \lambda^{n/2t} u (\lambda y) \,,
\]
for $\lambda > 0$. First, suppose that $u \in \dot{H}^1 (\mathbb{R}^n)$ is a nonnegative function satisfying (\ref{Normalization}).  Collecting together (\ref{DeficitEq}) and Lemmas \ref{ControlInfLm} and \ref{BoundingLm}, we deduce that there exist constants $K_1 (n,t)$, $\delta_1(n,t) > 0$ depending on $n$ and $t$ such that whenever $\delta_{GN} [u] \leq \delta_1$,
\[
\| u^{2t} - v^{2t} (\cdot - x_0) \|_1 \leq K_1 \delta_{GN} [u]^{1/2}
\]

Next, $\delta_{GN} [u]$ and $\| u \|_{2t}$ are both unchanged if $u$ is replaced by $u_\lambda$.  Thus, assuming only that $\| u \|_{2t} = \| v \|_{2t}$, we may choose a scale parameter $\lambda$ so that $u_\lambda$ satisfies (\ref{Normalization}).  In this case,
\[
\int_{\mathbb{R}^n} | \lambda^n u^{2t} (\lambda y) - v^{2t} (y) | \mathrm{d}y = \int_{\mathbb{R}^n} | u_\lambda^{2t} (y) - v^{2t} (y) | \mathrm{d}y \leq K_1 \delta_{GN} [u]^{1/2} \,.
\]
Changing variables once more,  and taking $\kappa = 1/\lambda$, we obtain
\[
\int_{\mathbb{R}^n} | u^{2t} (y) - \kappa^n v^{2t} (\kappa y) | \mathrm{d}y \leq K_1 \delta_{GN} [u]^{1/2} \,,
\]
i.e.
\[
\int_{\mathbb{R}^n} | u^{2t} (y) - v_\kappa^{2t} (y) | \mathrm{d}y \leq K_1 \delta_{GN} [u]^{1/2} \,.
\]
This conclude the proof of Theorem \ref{GN StabEst}.

\section{Appendix}

In the following, we provide detailed calculations of the formulas (\ref{Est12})-(\ref{Est17}). On $A_1$, we compute
\begin{align}
&\int_{A_1} \left( \int_{\mathbb{R}_+} \frac{|(w_u + \rho^2)^\gamma - (H+\rho^2)^\gamma|^{2^*}}{(w_u + \rho^2)^{m+n} (H + \rho^2)^{m+n}} \mathrm{d} \Omega (\rho) \right) \mathrm{d}x \text{, which by the Mean Value Theorem} \nonumber \\
 &= \int_{A_1} \left( \int_{\mathbb{R}_+} \frac{| \gamma (G+\rho^2)^{\gamma-1} (H-w_u)|^{2^*}}{(w_u + \rho^2)^{m+n} (H + \rho^2)^{m+n}} \mathrm{d} \Omega (\rho) \right) \mathrm{d}x
 \text{, for some $w_u < G < H$} \nonumber \\
 &\geq \gamma^{2^*} \int_{A_1} \left( \int_{\mathbb{R}_+} \frac{|H-w_u|^{2^*}}{(H+\rho^2)^{2^* + m + n}} \mathrm{d} \Omega (\rho) \right) \mathrm{d}x \,, \text{ substituting $\rho = H^{1/2} \theta$} \nonumber \\
 &= \gamma^{2^*} \int_{\mathbb{R}_+} (1+\theta^2)^{-(2^* + m + n)} \mathrm{d} \Omega (\theta) \int_{A_1} \frac{|H-w_u|^{2^*}}{H^{2^* + \frac{m+2n}{2}}} \mathrm{d}x \,. \nonumber \tag{\ref{Est12}}
\end{align}
Also, pointwise on $A_1$,
\begin{align}
 |u^{2t} - v^{2t}| &= \left| \frac{H^\frac{m+2n}{2} - w_u^\frac{m+2n}{2}}{w_u^\frac{m+2n}{2} H^\frac{m+2n}{2}} \right| \text{, which by the Mean Value Theorem} \nonumber \\
 &= \frac{\left|\frac{m+2n}{2} \tilde{G}^{\frac{m+2n}{2}-1} (H-w_u) \right|}{w_u^\frac{m+2n}{2} H^\frac{m+2n}{2}} \text{, for some $w_u < \tilde{G} < H$} \nonumber \\
 &\leq 2^\frac{m+2n}{2} (m+2n) \frac{|H-w_u|}{H^{\frac{m+2n}{2}+1}} \,, \text{ because $\tilde{G} > w_u > H/2$.} \nonumber \tag{\ref{Est13}}
\end{align}
Note that
\begin{equation}\tag{\ref{Eq1}}
\frac{m+2n}{2} = \frac{m + 2n}{2 \cdot 2^*} + \frac{(m+2n)(2^*-1)}{2 \cdot 2^*} \,.
\end{equation}
Thus,
\begin{align}
 \int_{A_1} |u^{2t}-v^{2t}| \mathrm{d}x \leq& 2^\frac{m+2n}{2} (m+2n) \int_{A_1} \frac{|H-w_u|}{H^{\frac{m+2n}{2}+1}} \mathrm{d}x \,, \text{ by (\ref{Est13})} \nonumber \\
 \leq& 2^\frac{m+2n}{2} (m+2n) \left( \int_{\mathbb{R}^n} H^{-\frac{m+2n}{2}} \mathrm{d}x \right)^{1 - 1/2^*} \left( \int_{A_1} \frac{|H-w_u|^{2^*}}{H^{\frac{m+2n}{2}+2^*}} \mathrm{d}x \right)^{1/2^*} \nonumber \\
 &\text{by (\ref{Eq1}) and Holder's Inequality} \nonumber \\
 \leq& \frac{2^\frac{m+2n}{2} (m+2n)}{\gamma} \left( \| v \|_{2t}^{2t} \right)^{1-1/2^*} \left( \int_{\mathbb{R}_+} (1+\theta^2)^{-(2^* + m + n)} \mathrm{d} \Omega (\theta) \right)^{-1/2^*} \cdot \nonumber \\
 & \left[ \int_{A_1} \left( \int_{\mathbb{R}_+} \frac{|(w_u + \rho^2)^\gamma - (H +\rho^2)^\gamma|^{2^*}}{(w_u + \rho^2)^{m+n} (H + \rho^2)^{m+n}} \mathrm{d} \Omega (\rho) \right) \mathrm{d}x \right]^{1/2^*} \text{, by (\ref{Est12}).} \nonumber \tag{\ref{Est14}}
\end{align}

On $A_2$, we compute
\begin{align}
& \int_{A_2} \left( \int_{\mathbb{R}_+}  \frac{|(H+\rho^2)^\gamma - (w_u+\rho^2)^\gamma|^{2^*}}{|w_u+\rho^2|^{m+n} |H+\rho^2|^{m+n}} \mathrm{d} \Omega (\rho) \right) \mathrm{d}x \nonumber \\
\geq& \int_{A_2} \left( \int_{\mathbb{R}_+} \frac{|(H+\rho^2)^\gamma - (\frac{H}{2}+\rho^2)^\gamma|^{2^*}}{|w_u+\rho^2|^{m+n} |H + \rho^2|^{m+n}} \mathrm{d} \Omega (\rho) \right) \mathrm{d}x \text{, which by the Mean Value Theorem} \nonumber  \\
=& \int_{A_2} \left( \int_{\mathbb{R}_+} \frac{|\gamma (K+\rho^2)^{\gamma-1} \frac{H}{2}|^{2^*}}{|w_u+\rho^2|^{m+n} |H + \rho^2|^{m+n}} \mathrm{d} \Omega (\rho) \right) \mathrm{d}x \text{, for some $H / 2 < K < H$} \nonumber \\
=& \left( \frac{\gamma}{2} \right)^{2^*} \int_{A_2} H^{2^*} \left( \int_{\mathbb{R}_+} \frac{(K+\rho^2)^{m+n-{2^*}} }{|w_u+\rho^2|^{m+n} |H + \rho^2|^{m+n}} \mathrm{d} \Omega (\rho) \right) \mathrm{d}x \nonumber \,. \tag{\ref{Est15}}
\end{align}

We will treat (\ref{Est15}) in two cases.  

\textit{Case 1: $2 < m+n \leq 4$.}  In this case, $m+n-2^* \leq 0$.   Thus,
\begin{align}
& \int_{A_2} H^{2^*} \left( \int_{\mathbb{R}_+} \frac{(K+\rho^2)^{m+n-{2^*}}}{|w_u+\rho^2|^{m+n} |H + \rho^2|^{m+n}} \mathrm{d} \Omega (\rho) \right) \mathrm{d}x \nonumber \\
\geq& \int_{A_2} H^{2^*} \left( \int_{\mathbb{R}_+} \frac{(H+\rho^2)^{m+n-2^*}}{|w_u+\rho^2|^{m+n} |H + \rho^2|^{m+n}} \mathrm{d} \Omega (\rho) \right) \mathrm{d}x \,, \text{ because $m+n-2^* \leq 0$} \nonumber \\
=& \int_{A_2} H^{2^*} \left( \int_{\mathbb{R}_+} |w_u+\rho^2|^{-(m+n)} |H + \rho^2|^{-2^*} \mathrm{d} \Omega (\rho) \right) \mathrm{d}x \,, \text{ letting $\rho = H^{1/2} \theta$} \nonumber \\
=& \int_{A_2} H^{-\frac{m+2n}{2}} \left( \int_{\mathbb{R}_+} \left| \frac{w_u}{H}+\theta^2 \right|^{-(m+n)} (1+\theta^2)^{-2^*} \mathrm{d} \Omega (\theta) \right) \mathrm{d}x \nonumber \\
\geq& \int_{A_2} H^{-\frac{m+2n}{2}} \left( \int_0^1  \left| \frac{w_u}{H}+\theta^2 \right|^{-(m+n)} 2^{-2^*} \mathrm{d} \Omega (\theta) \right) \mathrm{d}x \,, \text{ as $1 + \theta^2 \leq 2$ for $\theta \in [0,1]$} \nonumber \\
& \text{letting $\theta = \left( w_u / H \right)^{1/2} \tilde{\theta}$ and using the fact that $\left( H / w_u \right)^{1/2} \geq \sqrt{2}$ on $A_2$} \nonumber \\
\geq& 2^{-2^*} \int_{A_2} w_u^{-\frac{m+2n}{2}} \left( \int_0^{\sqrt{2}} (1+\tilde{\theta}^2)^{-(m+n)} \mathrm{d} \Omega (\tilde{\theta}) \right) \mathrm{d}x \nonumber \\
\geq& 2^{-2^*} \int_0^{\sqrt{2}} (1+\tilde{\theta}^2)^{-(m+n)} \mathrm{d} \Omega (\tilde{\theta})  \int_{A_2} w_u^{-\frac{m+2n}{2}} - H^{-\frac{m+2n}{2}} \mathrm{d}x \nonumber \\
\geq& 2^{-2^*} \int_0^{\sqrt{2}} (1+\tilde{\theta}^2)^{-(m+n)} \mathrm{d} \Omega (\tilde{\theta}) \int_{A_2} |u^{2t} - v^{2t}| \mathrm{d}x \nonumber \,. \tag{\ref{Est16}}
\end{align} 

\textit{Case 2: $m+n \geq 4$.} In this case, $m+n-2^* \geq 0$.  Thus,
\begin{align}
 & \int_{A_2} H^{2^*} \left( \int_{\mathbb{R}_+} \frac{(K+\rho^2)^{m+n-{2^*}}}{|w_u+\rho^2|^{m+n} |H + \rho^2|^{m+n}} \mathrm{d} \Omega (\rho) \right) \mathrm{d}x \nonumber \\
 \geq& \int_{A_2} H^{2^*} \left( \int_{\mathbb{R}_+} \frac{(\frac{H}{2}+\rho^2)^{m+n-2^*}}{|w_u+\rho^2|^{m+n} |H + \rho^2|^{m+n}} \mathrm{d} \Omega (\rho) \right) \mathrm{d}x \,, \text{ letting $\rho = H^{1/2} \theta$} \nonumber \\
 =& \int_{A_2} H^{-\frac{m+2n}{2}} \left( \int_{\mathbb{R}_+} \frac{(\frac{1}{2}+\theta^2)^{m+n-2^*}}{|\frac{w_u}{H}+\theta^2|^{m+n} |1 + \theta^2|^{m+n}} \mathrm{d} \Omega (\theta) \right) \mathrm{d}x \nonumber \\
 \geq& \int_{A_2} H^{-\frac{m+2n}{2}} \left( \int_0^1 \frac{(\frac{1}{2})^{m+n-{2^*}}}{|\frac{w_u}{H}+\theta^2|^{m+n} 2^{m+n}} \mathrm{d} \Omega (\theta) \right) \mathrm{d}x \,, \text{ as $1 + \theta^2 \leq 2$ for $\theta \in [0,1]$} \nonumber \\
 & \text{letting $\tilde{\theta} = \left( w_u / H \right)^{1/2} \theta$ and using the fact that $\left( H / w_u \right)^{1/2} \geq \sqrt{2}$} \nonumber \\
 \geq& 2^{-2(m+n) + 2^*} \int_{A_2} w_u^{-\frac{m+2n}{2}} \left( \int_0^{\sqrt{2}} (1+\tilde{\theta^2})^{-(m+n)} \mathrm{d} \Omega (\tilde{\theta}) \right) \mathrm{d}x \nonumber \\
 \geq& 2^{-2(m+n) + 2^*} \int_0^{\sqrt{2}} (1+\tilde{\theta^2})^{-(m+n)} \mathrm{d} \Omega (\tilde{\theta}) \int_{A_2} |u^{2t} - v^{2t}| \mathrm{d}x \nonumber. \tag{\ref{Est17}}
\end{align}

\end{document}